\documentclass[12pt,reqno]{amsart}

\usepackage[margin=4cm]{geometry}

\usepackage[foot]{amsaddr}
\usepackage{amsmath}
\usepackage{amssymb}
\usepackage{amsthm}
\usepackage{graphicx}
\usepackage{float}
\usepackage{color}
\usepackage{ulem}
\usepackage{dsfont}
\usepackage{epstopdf}
\usepackage{hyperref}
\usepackage{appendix}
\usepackage{comment}
\allowdisplaybreaks[4]
\numberwithin{equation}{section}

\newtheorem{theorem}{Theorem}[section]
\newtheorem{corollary}[theorem]{Corollary}

\newtheorem{claim}{Claim}

\newtheorem{lemma}[theorem]{Lemma}
\newtheorem{proposition}[theorem]{Proposition}
\newtheorem{conjecture}{Conjecture}

\theoremstyle{definition}
\newtheorem{definition}[theorem]{Definition}
\newtheorem{remark}[theorem]{Remark}

\newcommand{\ep}{\varepsilon}

\def\Inf{\operatornamewithlimits{inf\vphantom{p}}}

\title[On the evolution of slow dispersal]
      {On the evolution of slow dispersal in multi-species
communities}

\thanks{$^1$ Department of Mathematics, University of Miami, Coral Gables, FL 33124, USA}

\thanks{$^2$ Department of Mathematics, Ohio State University, Columbus, OH 43210, USA}

\thanks{RSC is partially supported by NSF grant DMS-1853478; KYL is partially supported by NSF grant DMS-1853561}

\author[R.S. Cantrell and K.-Y. Lam]
{Robert Stephen Cantrell$^1$ and King-Yeung Lam$^2$}


\begin{document}

\maketitle

\begin{abstract}
For any $N \geq 2$, we show that there are choices of diffusion rates $\{d_i\}_{i=1}^N$ such that for $N$ competing species which are ecologically identical and having distinct diffusion rates, the slowest disperser is able to competitive exclude the remainder of the species. In fact, the choices of such diffusion rates is open in the Hausdorff topology. Our result provides some evidence in the affirmative direction regarding the conjecture by Dockery et al. in 1998.
The main tools include Morse decomposition of the semiflow, as well as the theory of normalized principal Floquet bundle for linear parabolic equations. A critical step in the proof is to establish the smooth dependence of the Floquet bundle on diffusion rate and other coefficients, which may be of independent interest. 
\end{abstract}

\section{Introduction}

Many organisms adapt to the surrounding environment through their dispersal behavior. It is important to determine the circumstances in which organisms 
modify their dispersal strategies
under the driving forces of evolution.
In a pioneering paper, Hastings introduced the point of view of studying the effect of individual factors on the evolution of dispersal independently, using mathematical modeling \cite{Hastings1983}. By analyzing the outcome of invasion between two competing species, assuming they are identical except for their dispersal rates, Hastings showed that passive diffusion is selected against in an environment that is constant in time, but varies in space. Subsequently, Dockery et al. \cite{Dockery1998} 
refined Hastings' findings via
a more explicit Lotka-Volterra model. They showed that it is impossible for two or more ecologically identical species, moving randomly at different rates, to coexist at an equilibrium. When the number of species is equal to two, they determined the global dynamics of the competition system completely by
demonstrating the faster disperser is always driven to extinction by the slower disperser, regardless of initial conditions. 

The work of Hastings and Dockery et al. have been highly influential in prompting advances in both mathematical and biological aspects of the evolution of dispersal. In \cite{Altenberg} Altenberg showed a reduction principle, which says that the growth bounds for certain class of linear operators exhibit monotone dependence on the mixing coefficient. This principle gives a mathematical explanation of the relative proliferative advantage of slower dispersers in a static, but spatially varying environment. 
It is also demonstrated that slow dispersal might not be advantageous if time-periodicity is included \cite{Hutson2001}.

The theory of habitat selection can also explain the evolution of slow dispersal among passive dispersers. As observed by Hastings, passive diffusion transports individuals from more favorable locations to less favorable ones in an average sense,  rendering passive diffusion to be selected against. The picture is different, however, if the dispersal is conditional on the local environment. An important class of dispersal strategies consists of ones enabling a population to become perfectly aligned with the heterogeneous resource distribution, thus achieving the so called ideal free distribution \cite{Fretwell}. In this circumstance, it is shown that such a dispersal strategy is selected for, in the sense that it is both an evolutionarily stable strategy and a neighborhood invader strategy. See \cite{Aver2012,Cantrell2010,Cho2013,Korobenko2014} for results on reaction-diffusion models; and \cite{Cantrell2012a,Cantrell2012b,Cantrell2020,Cosner2012} for results in other modeling settings.

The work of Hastings and Dockery et al. has also stimulated substantial mathematical analysis of competition models involving two species. We mention the work of \cite{He2016,Hutson2001,Lam2012,Lou2006} for passive dispersal, and \cite{ALL,Cantrell2007,Chen2008,Chen2012,Lam2014a,Lam2014b} 
for conditional dispersal. An interesting application concerns the evolution of dispersal in stream populations, which are subject to a uni-directional drift \cite{Lutscher2005,Potapov2014}. It has been shown that in some circumstances, faster dispersal is sometimes selected for \cite{Lou2014,Lou2015}. See also \cite{Hao2020,Lam2015,Lou2019}. We also mention the work \cite{Jiang2019} on the evolution of dispersal in phytoplankton populations, where individuals compete non-locally for sunlight.

Most of the existing results are restricted to the case when the number of species is equal to two. In this case, the theory of monotone dynamical systems \cite{Hsu1996,LamMunther,Smith1995} can be applied to determine the global dynamics of the competition system. Results for three or more competing species are relatively rare \cite{Cantrell1993,Cantrell1997,Dancer1995,Dancer2012,FengRuan,LouMunther,Mazari2019}, 
and the question of global dynamics remains an open and challenging problem. In the following, we will address two conjectures of Dockery et al. concerning a model involving $N$ competing species, which are identical except for the passive dispersal rates.

\subsection{Two conjectures of Dockery et al.}

The following Lotka-Volterra model of $N$ competing species, which are subject to passive dispersal, was introduced by Dockery et al. \cite{Dockery1998}. 
\begin{equation}\label{eq:1.1}
\begin{cases}
\partial_t u_i(x,t) = d_i \Delta u_i(x,t) + u_i(x,t) \left[ m(x) - \sum_{j=1}^N u_j(x,t)\right]  \\
\hspace{6cm} \text{ for }x \in \Omega,\, t>0,\, 1 \leq i \leq N. \,\,\,\\
\partial_\nu u_i(x,t) = 0  \hfill \text{ for }x \in \partial\Omega,\, t>0,\, 1 \leq i \leq N.\\
\end{cases}
\end{equation}
These $N$ species are assumed to be identical except for their dispersal rates $0<d_1<...<d_N$.
Here $\Omega$ is a bounded domain in $\mathbb{R}^n$ with smooth boundary $\partial\Omega$, $\Delta = \sum_{j=1}^n \partial_{x_j x_j}$ is the Laplacian operator, $\partial_\nu$ is the outer-normal derivative on $\partial\Omega$. We also assume, {for some $\beta \in (0,1)$ and $\epsilon'>0$, 
\begin{equation}\label{eq:MmM}
m(x) \in C^{\beta+\epsilon'}(\overline\Omega) \,\,\text{ is non-constant, and } \,\, \int_\Omega m\,dx \geq 0,
\end{equation}
where $C^{\beta + \epsilon'}(\overline\Omega)$ is the H\"{o}lder space. The function $m$ and the exponent $0 < \beta <1$ will be fixed throughout this paper.}
In the following we denote 
\begin{equation}\label{eq:equilibria}
    E_i = (0,..., 0,\theta_{d_i},0,...,0)\quad \text{ for }\quad 1 \leq i \leq N, \quad \text{ and }\quad E_0=(0,...,0)
\end{equation} to be the corresponding equilibria of \eqref{eq:1.1}, where for $d>0$
the function
$\theta_d(x)$ denotes the unique positive solution of 
\begin{equation}\label{eq:theta}
    d \Delta \theta_d + \theta_d(m(x) - \theta_d) = 0 \,\,\text{ in }\Omega,\,\text{ and }\,\ \partial_\nu \theta_d = 0\,\, \text{ on }\partial\Omega.
\end{equation}

In case $N=2$, Dockery et al. obtained a complete description of the dynamics of \eqref{eq:1.1} by applying the abstract tools of monotone dynamical systems.
\begin{theorem}[{\cite[Lemmas 3.9 and 4.1]{Dockery1998}}] \label{thm:Dockery}
Suppose $N=2$ and $d_1<d_2$. Then every positive solution of \eqref{eq:1.1} converges to the equilibrium $(\theta_{d_1},0)$. 
Furthermore,   a Morse decomposition for {$\textup{Inv}\,K^+$} is given by 
$$
M(1) = \{E_1\}\quad M(2) = \{E_2\},\quad M(3) = \{E_0\}.
$$
Here {$\textup{Inv}\,K^+$} denotes 
 the maximal bounded invariant set of the dynamical system generated in $K^+=\{(u_i)_{i=1}^N \in [C(\overline\Omega)]^N:\, u_i \geq 0\}$ under \eqref{eq:1.1}.
\end{theorem}
Roughly speaking, we say that $\{M(1),M(2),M(3)\}$ is a Morse decomposition of {$\textup{Inv}\,K^+$} if {every} bounded trajectory $\gamma(t)$ converges to some equilibrium, and that, if $\gamma(t)$ is defined for $t \in \mathbb{R}$, then $\gamma(\infty) \subset M(i)$ and $\gamma(-\infty) \subset M(j)$ for some $i < j$. The precise definition of a Morse decomposition will be given in Subsection \ref{subsec:1.4}, after some related dynamical system concepts are introduced.

When $N \geq 3$, it is conjectured in \cite{Dockery1998} that the slowest disperser continues to win the competition.

\begin{conjecture}\label{conj:A}
Let $N\geq 1$ and  $0 < d_1<...<d_N$. Then the equilibrium $E_1=(\theta_{d_1},0,...,0)$ is globally asymptotically stable among all positive solutions of \eqref{eq:1.1}. 
\end{conjecture}
Another version of the conjecture, also due to Dockery et al., can be formulated in terms of the concept of Morse Decomposition.
\begin{conjecture}\label{conj:B}
Let $N\geq 1$ and $0 < d_1<...<d_N$. Then a Morse decomposition for {$\textup{Inv}\,K^+$} is given by 
$$
M(i) = \{E_i\}\quad \text{ for }1 \leq i \leq N, \quad \text{ and }\quad M(N+1) = \{E_0\}.
$$
\end{conjecture}

Define $\mathcal{D}$ to be the collection of all finite sets of positive real numbers such that Conjecture \ref{conj:B} holds; i.e.
$$
\mathcal{D} = \cup_{{N=1}}^\infty \{(d_i)_{i=1}^N\,:\, \text{Conjecture \ref{conj:B} holds.}\}.
$$
By the result of Dockery et al., the family $\mathcal{D}$ contains all singleton and doubleton sets of positive numbers. Can we say more about $\mathcal{D}$?

We first observe that Conjecture \ref{conj:B} implies Conjecture \ref{conj:A}, for any $N$.
\begin{proposition}\label{prop:1.1}
Let $0<d_1<d_2<...<d_N$ be given. If $(d_i)_{i=1}^N \in \mathcal{D}$, then 
every interior trajectory of \eqref{eq:1.1} converges to $E_1$. 
\end{proposition}
Moreover, if $N=3$, the two conjectures are equivalent.
\begin{proposition}\label{prop:1.2}
Let $0<d_1<d_2<d_3$ be given. Then $(d_i)_{i=1}^3 \in \mathcal{D}$ if and only if  
every interior trajectory of \eqref{eq:1.1} converges to $E_1$. 
\end{proposition}

The goal of this paper is to prove the following stability result, which provides 
a step towards an affirmative answer to Conjectures \ref{conj:A} and \ref{conj:B}. 
\begin{theorem}\label{thm:main}
The collection $\mathcal{D}$ is open in the space of finite sets relative to the Hausdorff metric. 
\end{theorem}
{We recall that the Hausdorff metric is given by
$$
\textup{dist}_H (A,B) = \max\left\{ \sup_{x \in A} \Inf_{y \in B} |x-y|,\,\,\, \sup_{y \in B} \Inf_{x \in A} |x-y| \right\}
$$
for any two non-empty subsets $A, B$ of  $\mathbb{R}$. In particular, the collection of finite subsets of $\mathbb{R}_+$ forms a metric space under the Hausdorff metric. }

As a corollary, we obtain some global stability results for \eqref{eq:1.1} with no restriction on the number of species $N$. 
\begin{corollary}\label{cor:B}
Given $0<\hat{d}_1<\hat{d}_2$, there exists $\ep>0$ such that for any integer $N \geq 3$ and any $(d_i)_{i=1}^N$ such that $0<d_1<d_2<...<d_N$, {if
$$
(d_i)_{i=1}^N \subseteq (\hat{d}_1 - \ep, \hat{d}_1+\ep) \cup (\hat{d}_2-\ep,\hat{d}_2+\ep), 
$$
$($i.e. each $d_i$ is close to either $\hat{d}_1$ or $\hat{d}_2)$}  
then for the problem \eqref{eq:1.1} of $N$ species with diffusion rates $(d_i)_{i=1}^N$, every positive solution converges to the equilibrium $E_1=(\theta_{d_1},0,...,0)$ as $t \to \infty$.  
\end{corollary}
\begin{proof}
{Given $0 < \hat{d}_1<\hat{d}_2$, Theorem \ref{thm:Dockery} says that the sets 
$$
A_1 = \{\hat{d}_1\},\quad A_2 = \{\hat{d}_2\}, \quad \text{ and }\quad A_3 = \{\hat{d}_1,\hat{d}_2\}
$$
belong to $\mathcal{D}$. Moreover, Theorem \ref{thm:main} says that they
are interior points of the collection $\mathcal{D}$ in the Hausdorff sense. Hence, there exists $\ep$ so that any sets within $\ep$ distance from any one of $A_i$ belongs to $\mathcal{D}$. By assumption of the corollary, the set $(d_i)_{i=1}^N$ is within $\ep$ distance from at least one $A_i$, so it belongs to $\mathcal{D}$. Finally, it follows from 
Proposition \ref{prop:1.1} that for the problem \eqref{eq:1.1} of $N$-species with diffusion rates $(d_i)_{i=1}^N$, the equilibrium solution $E_1$ attracts all positive solutions.}
\end{proof}

Finally, we also mention that the dynamics of an arbitrary number of competing species was considered in
the paper \cite{Cantrell2012a} in the context of patch models, which are  discrete-in-space versions of \eqref{eq:1.1}. When at least one of the patches is a sink (which is equivalent to $m(x)$ changing sign in the reaction-diffusion context), it was shown that the zero disperser can competitively exclude all other species, by the construction a Lyapunov function.

\subsection{Definitions}\label{subsec:1.4}

{For $L \in \mathbb{N}$, let $X_L=[C(\overline\Omega)]^L$ be the Banach space with norm 
$$\|u\| = \max\limits_{1 \leq i \leq L} \sup\limits_{x\in\Omega} |u_i|,$$
and let $K_L^+$ be the cone of non-negative functions in $X_L$. For simplicity, we will sometimes suppress the subscript $L$ and simply write $X$ and $K^+$ when it does not cause confusion.}

The Neumann Laplacian operator $-\Delta$ is a sectorial operator with domain 
$$
D(-\Delta)=\bigcap\limits_{r>1}\{u \in W^{2,r}(\Omega):\, \Delta u \in C(\overline\Omega),\,\text{ and }\, \partial_\nu u = 0 \,\,\text{ on }\partial\Omega\};
$$
and we denote the fractional power of $-\Delta$ by $(-\Delta)^{\xi}$ for some $0 < \xi < 1$ (see \cite[Ch. 2]{Lunardi}).
It is a standard fact that the reaction-diffusion system \eqref{eq:1.1} generates a semiflow in $X$, which we will denote here by $\Psi: [0,\infty)\times X$, i.e., for the solution $u(x,t)$ of \eqref{eq:1.1} it holds that
$$
u(\cdot,t+t_0) = \Psi(t, u(\cdot,t_0)) \quad \text{ for }t , t_0  \geq 0.
$$

We say that a function $\gamma:\mathbb{R} \to X$ is a full trajectory if 
$$
\gamma(t + t_0) = \Psi(t, \gamma(t_0)) \quad \text{ for all }t \geq 0 \text{ and } t_0 \in \mathbb{R}.
$$

A subset $A$ of $X$ is an invariant set if every $a \in A$ lies on a full trajectory $\gamma(t)$ such that $\{\gamma(t): t\in \mathbb{R} \} \subset A$. 
Let {$\textup{Inv}\,K^+$} denote the maximal bounded invariant set in $K^+$ under \eqref{eq:1.1}. It is not difficult to see that {{$\textup{Inv}\,K^+$}\,} is compact, and attracts every trajectory in $K^+$. 

Recall also that the $\omega$- and $\alpha$-limit sets of a point $u_0 \in K^+$ are given by 
$$
\begin{cases}
\omega(u_0)= \{\tilde{u} \in X: \, {\Psi}(t_j,u_0) \to \tilde u\quad \text{ for some }  t_j \to \infty.\},\\
\alpha(u_0)= \{\tilde{u} \in X: \, {\Psi}(t_j,u_0) \to \tilde u\quad \text{ for some }  t_j \to -\infty.\},
\end{cases}
$$
where the latter is well-defined if and only if $u$ lies on a full trajectory.

Next, we define the concept of Morse decomposition, which is relevant in considering the global dynamics of \eqref{eq:1.1}. We say that a finite and ordered collection of disjoint compact invariant subsets of {$\textup{Inv}\,K^+$}, 
$$
\{M(i) \subset \text{Inv}\,K^+\,:\, 1\leq i \leq m\},
$$
is a Morse decomposition if, for every $u_0 \in K^+ \setminus \cup_{i=1}^m M(i)$ with bounded trajectory, there exists $i$ with $1\leq i \leq m$ such that $\omega(u_0) \subset M(i)$, and if $u$ lies on a full trajectory, then there exists $j$ such that $i < j \leq m$ and $\alpha(u_0) \subset M(j).$

\subsection{Proofs of Propositions \ref{prop:1.1} and \ref{prop:1.2}}
First, we recall the statement of \cite[Lemma 3.9]{Dockery1998}. (See also Lemma \ref{lem:3.9} in this article.)
\begin{lemma}\label{lem:3.9'}
Fix $0<d_1<...<d_N$. 
For any $u_0 \in \textup{Int}\,K^+$, if the trajectory $\Psi(t,u_0)$ converges to an equilibrium, i.e. 
$$
\lim_{t \to \infty} \Psi(t,u_0) = E_i \quad \text{ for some }i\in \{0,1,2,3,...,N\},
$$
then necessarily $i=1$; i.e. $\Psi(t,u_0) \to E_1$. Here $E_i$ is defined in \eqref{eq:equilibria}.
\end{lemma}

\begin{proof}[Proof of Proposition \ref{prop:1.1}]
Suppose $({d}_i)_{i=1}^N \in \mathcal{D}$, then the system \eqref{eq:1.1} admits a Morse decomposition where the Morse sets consist of the $(N+1)$ equilbria. Hence, every internal trajectory converges to an equilibrium $E_i$ (see \eqref{eq:equilibria}). By Lemma \ref{lem:3.9'}, it can only converge to $E_1$.  
\end{proof}
\begin{proof}[Proof of Proposition \ref{prop:1.2}]
It suffices to show the converse. Suppose $0<d_1<d_2<d_3$ are given such that all interior trajectories of \eqref{eq:1.1} converge to $E_1$. We need to show that 
$$
M(1)=\{E_1\},\quad M(2) = \{E_2\}, \quad M(3) = \{E_3\}, \quad M(4) = \{(0,0,0)\}
$$
is a Morse decomposition of the semiflow. By \cite[Theorem 3.2 and Remark 4.6]{Hirsch2001}, it suffices to show that $\{M(j)\}_{j=1}^4$ forms an acyclic covering of $\displaystyle \cup_{u_0 \in K^+} \omega(u_0)$.

First we show that $\cup_{i=1}^4 M(i)=\{E_1,E_2,E_3,E_0\}$ is a covering, namely, $\cup_{u_0 \in K^+} \omega(u_0) \subset \{E_1,E_2,E_3,E_0\}.$ For the trajectories starting at $u_0 \in K^+ \setminus (\textup{Int}\,K^+)$, by strong maximum principle, it either enters the $\textup{Int}\,K^+$ for all $t>0$, or there is at least one component that is identically zero for all $t>0$. In the first case, the trajectory also converges to $E_1$. In the second case, the system reduces to the two-species case, so that {Theorem \ref{thm:Dockery} applies and the solution converges to $E_{i_0}$,  where $1 \leq i_0\leq 3$ is the smallest integer such that the $i_0$-th} component of $u_0$ is non-zero. 

Next, we show that $\{M(j)\}_{j=1}^4$ is acyclic, i.e. there is no cycle of fixed points. Indeed, if $E_i$ is chained to $E_j$ (i.e. there is a full trajectory connecting from $E_i$ to $E_j$), then 
the trajectory is a positive solution of either the full three-species system, or one of the two-species subsystems. In either case, we have $i > j$. This shows acyclicity. It therefore follows from {\cite[Theorem 3.2 and Remark 4.6]{Hirsch2001}} that any compact internally chain transitive set is an equilibrium point. Since any omega (resp. alpha) limit set is internally chain transitive, it can only be one of the $E_i$'s. The proof of the proposition is completed.
\end{proof}

\subsection{Setting up the proof of Theorem \ref{thm:main}}\label{subsec:1.6}

Suppose {$N_0 \geq 1$} and a finite increasing sequence $(\hat{d}_k)_{k=1}^{N_0} \in \mathcal{D}$ are given. Consider, for a small $\ep \in (0, \hat{d}_1/2)$, an arbitrary $N \geq N_0$ and arbitrary increasing sequence $(d_i)_{i=1}^N$ such that \begin{equation}\label{eq:haus}
\textup{dist}_H((d_i)_{i=1}^N,(\hat{d}_k)_{k=1}^{N_0})<\ep.
\end{equation}
In other words,
\begin{equation}\label{eq:pert}
{0<\frac{\hat{d}_1}{2}<d_1<...<d_N} \quad \text{ and }\quad (d_i)_{i \in I_k} \subset ( \hat{d}_k-{{\ep}}, \hat{d}_k+{{\ep}}) 
\end{equation}
for some partition $\{I_k\}_{k=1}^{N_0}$ of $\{1,2,\dots,N\}$. Note that $I_k$ is non-empty for all $k$, due to \eqref{eq:haus}. We introduce three closely related dynamical systems.

Let $\Phi:[0,\infty) \times {K^+_{N_0} \to K^+_{N_0}}$ be the semiflow operator generated by  the unperturbed problem of ${N_0}$ species:
\begin{equation}\tag{$\hat{P}_0$}\label{eq:p0}
\begin{cases}
\partial_t \hat{U}_k(x,t) = \hat{d}_k \Delta \hat{U}_k(x,t) + \hat{U}_k(x,t) \left[ m(x) - \sum_{j=1}^{N_0} \hat{U}_j(x,t)\right] &\text{ for }1 \leq k \leq {N_0}\\ 
\partial_\nu \hat{U}_k(x,t) = 0 &\text{ for }1 \leq k \leq {N_0}
\end{cases}
\end{equation}
Let $\varphi:[0,\infty)\times {K^+_N \to K^+_N}$ be the semiflow operator generated by the unperturbed problem of $N$ species (with repeated diffusion rates):
\begin{equation}\tag{$P_0$}\label{eq:P0}
\begin{cases}
\partial_t u_i(x,t) = \hat{d}_k \Delta u_i(x,t) + u_i(x,t) \left[ m(x) - \sum_{j=1}^N u_j(x,t)\right] &\text{ for }i \in I_k,\, 1 \leq k \leq {N_0},\\ 
\partial_\nu u_i(x,t) = 0 &\text{ for }i \in I_k,\, 1 \leq k \leq {N_0}.
\end{cases}
\end{equation}
Let $\varphi_\ep:[0,\infty)\times {K^+_N \to K^+_N}$ be the semiflow operator generated by the perturbed problem of $N$ species (with distinct diffusion rates):
\begin{equation}\tag{$P_\ep$}\label{eq:Pep}
\begin{cases}
\partial_t u_i(x,t) = d_i \Delta u_i(x,t) + u_i(x,t) \left[ m(x) - \sum_{j=1}^N u_j(x,t)\right] &\text{ for }1 \leq i \leq N,\\ 
\partial_\nu u_i(x,t) = 0 &\text{ for }1 \leq i \leq N.
\end{cases}
\end{equation}
Then, define the projection
$\mathcal{P}: \mathbb{R}^N \to \mathbb{R}^{N_0}$ by
\begin{equation}\label{eq:P}
[\mathcal{P}(y_1,...,y_N)]_k = \sum_{i\in I_k} y_i \quad \text{ for }1 \leq k \leq {N_0}.
\end{equation}
and denote $U = \mathcal{P}u$, i.e.
\begin{equation}\label{eq:agg}
U_k:= \sum_{i \in I_k} u_i \quad \text{ for }1 \leq k \leq {N_0}.
\end{equation}
\begin{remark}\label{rmk:1}
Note that $\Phi(t, \mathcal{P}u_0) = \mathcal{P} \varphi(t,u_0)$ for all $t\geq 0$ and $u_0 \in K_N^+$.
\end{remark}

\subsection{Outline of the proof}
Let $(\hat{d}_k)_{k=1}^{N_0}$ and $(d_i)_{i=1}^N$ be two finite subsets of $\mathbb{R}_+$, which are close in Hausdorff topology such that  
$(\hat{d}_k)_{k=1}^{N_0} \in \mathcal{D}$. We need to show that $(d_i)_{i=1}^N \in \mathcal{D}$ by examining the semiflow generated by \eqref{eq:Pep}. The strategy of our proof is to first obtain a rough Morse decomposition of the flow of \eqref{eq:Pep} by relating it to \eqref{eq:p0}. This is based on the existence of a complete Lyapunov function for the unperturbed semiflow $\Phi$ corresponding to the Morse decomposition (Section \ref{sec:2}),  and some {\it a priori} parabolic estimates that imply uniform continuity of the intermediate and perturbed semiflows (Section \ref{sec:3}). Then the rough Morse decomposition implies that every interior trajectory of the perturbed semiflow is ultimately dominated by the group of slowest dispersers whose diffusion rates are in a neighborhood of $\hat{d}_1$ (Section \ref{sec:4}). In Section \ref{sec:5} and the Appendix, we {recall} the notion of normalized principal Floquet bundle, 
which is a generalization of the notion of principal eigenvalue for elliptic or periodic-parabolic problems and {establish} its smooth dependence with respect to the coefficients of the linear parabolic problem. This is the main technical tool to refine the Morse decomposition and complete the proof of the main theorem (Section \ref{sec:6}).
We believe this tool will also be useful in the study of dynamics of general reaction-diffusion systems which are not necessarily of Lotka-Volterra type; see, e.g. \cite{Cantrell2020b}. 
Some concluding remarks are presented in Section \ref{sec:7}.

\section{The complete Lyapunov function for the unperturbed semiflow $\Phi$}\label{sec:2}

Since $(\hat{d}_k)_{k=1}^{N_0} \in \mathcal{D}$, i.e. the semiflow $\Phi$ genereated by \eqref{eq:p0} admits a Morse decomposition $\{M(k)\}_{k=1}^{{N_0}+1}$, the classical theorem due to Conley \cite[Chapter II, Result 6B]{ConleyCBMS} (see also
\cite{Hurley1998,Patrao2007,Rybakowski1987} and \cite[Remark 1]{Dockery1998}) guarantees the existence of a continuous function $V: \mathcal{U}' \to \mathbb{R}$,
in some  neighborhood $\mathcal{U}'$ of {$\textup{Inv}\,K^+$} relative to $K^+$,
with the following properties: 
\begin{itemize}
    \item $M(k) \in V^{-1}(k)$ for each $k=1,...,{N_0}+1$,
    \item  If $\Phi([0,T],U_0) \subset [\mathcal{U}' \setminus \cup_{k=1}^{{N_0}+1} M(k)]$, then
    \begin{equation}\label{eq:decreasing}
    V(U_0) > V(\Phi(t,U_0)) \quad \text{ for all }t\in (0,T].
    \end{equation}
\end{itemize}
By Remark \ref{rmk:1}, {the function $V\circ \mathcal{P}$ is a Lyapunov function of the semiflow $\varphi$, which is generated by \eqref{eq:P0}.} It will be the main tool allowing us to control and compare the dynamics generated by the three semiflows given in Subsection \ref{subsec:1.6}.  In the following, we recall \cite[Lemma 4.4]{Dockery1998}. 
\begin{lemma}\label{lem:1.5}
For given $0< \hat{d}_1<...<\hat{d}_{N_0}$, consider the semigroup operator $\Phi$ generated by the problem \eqref{eq:p0}. 
For any $r>0$ and $\mu>0$, there exist $T>0$ and a neighborhood $\mathcal{U}$ of {$\textup{Inv}\,K^+$}  contained in $\mathcal{U}'$, 
such that if  $\Phi(t,{U}_0)$ is a solution of \eqref{eq:p0} such that
$$
\Phi([0,t], U_0) \subset \mathcal{U} \setminus \left[\cup_{k=1}^{{N_0}+1} B_r(M(k)) \right] \quad \text{ for some }t \geq T,
$$
then
$$
V(U_0) - V(\Phi(t,U_0)) > \mu.
$$
\end{lemma}
\begin{proof}
Let $r>0$ and $\mu>0$ be given. We first prove that there exists a neigborhood $\mathcal{U}$ of {$\textup{Inv}\,K^+$} such that
\begin{equation} \label{eq:UU}
\tilde{\mu}:=\inf \left[V(U_0)- V(\Phi(1,U_0)\right] >0,
\end{equation}
where the infimum is taken over all initial data $U_0$ satisfying $$\Phi([0,1],U_0) \subset \mathcal{U} \setminus [\cup_{k=1}^{{N_0}+1} B_r({M}(k))].$$ 
Suppose to the contrary that \eqref{eq:UU} fails for every neighborhood $\mathcal{U}$ of {$\textup{Inv}\,K^+$}, then
there exist sequences $\{U^n\} \subset \mathcal{U}'$ and $\{\mu_n\} \subset (0,\infty)$ such that
$$
\Phi([0,1],U^n) \subset \mathcal{U}'\setminus [\cup_{k=1}^{{N_0}+1} B_r({M}(k))], \quad  \text{dist}(U^n,\text{Inv}\,K^+) \to 0,\quad \mu_n \to 0.
$$
and
\begin{equation}
    V(U^n) - V(\Phi(1,U^n)) \leq \mu_n.
\end{equation}
By the compactness of {$\textup{Inv}\,K^+$}, we may pass to a subsequence so that 
$$
U^n \to \overline U \in \textup{Inv}\,K^+ \setminus [\cup_{k=1}^{{N_0}+1}B_r(M(k))].
$$
By continuous dependence on initial data on the compact time interval $[0,1]$, we have 
$\Phi([0,1],\overline{U}) \subset  \textup{Inv}\,K^+ \setminus [\cup_{k=1}^{{N_0}+1}B_r(M(k))]$, and hence
\begin{align*}
    0 &< V(\overline U) - V(\Phi(1,\overline U))=\lim\limits_{n \to \infty}(V(U^n) - V(\Phi(1,U^n))\leq \lim\limits_{n \to \infty} \mu_n =0,
\end{align*}
a contradiction. This shows the existence of a
neighborhood $\mathcal{U}$ of {$\textup{Inv}\,K^+$} and a positive number
$\tilde{\mu}>0$ such that \eqref{eq:UU} holds.

Finally, observe that for given $\mu>0$, if $\mu \in (0,\tilde{\mu}]$, then we can take $T=1$ and \eqref{eq:UU} implies the desired conclusion. In case $\mu > \tilde{\mu}$, it suffices to choose $T$ to be an integer such that $\tilde{\mu}T\geq  \mu$. 
\end{proof}

\section{Uniform continuity of the intermediate semiflow $\varphi$ and perturbed semiflow  $\varphi_\ep$}\label{sec:3}


Recall that $\varphi$ is the semiflow generated by \eqref{eq:P0} with diffusion rates $(\hat{d}_k)_{k=1}^{N_0}$, and $\varphi_\ep$ is the semiflow generated by \eqref{eq:Pep} with diffusion rates $(d_i)_{i=1}^N$ satisfying $N \geq N_0$ and \eqref{eq:pert}. The purpose of this section is to establish some parabolic estimates and show that the trajectories of $\varphi$ and $\varphi_\ep$ stay close in any finite time interval.
(In the following, $\|\cdot\| = \| \cdot\|_{C(\overline\Omega)}$ or $\|\cdot\|_{[C(\overline\Omega)]^n}$ for some $n$, unless otherwise specified.)

\begin{lemma}\label{lem:L1}
Let $(u_i)_{i=1}^N$ be a non-negative solution of \eqref{eq:Pep}  or \eqref{eq:P0}, 
then 
$$
\sup_{t \geq 0}\sum_{i=1}^N \|u_i(\cdot,t)\|_{L^1(\Omega)} \leq \max\left\{\sum_{i=1}^N\| u_i(\cdot,0)\|_{L^1(\Omega)},|\Omega|\sup_\Omega m\right\}, 
$$
and
$$
\limsup_{t \to \infty}\sum_{i=1}^N\| u_i(\cdot,t)\|_{L^1(\Omega)} < 2|\Omega|\sup_\Omega m.
$$
In particular, the set $\mathcal{N}$, given by 
\begin{equation}\label{eq:N}
\mathcal{N} = \{u \in K^+:\, \sum_{i=1}^N \|u_i\|_{L^1(\Omega)} < 2 |\Omega| \sup_\Omega m\},
\end{equation}
is open in $X$ and 
 forward-invariant with respect to both \eqref{eq:P0} and \eqref{eq:Pep}, and hence contains the respective maximal bounded invariant sets {$\textup{Inv}\,K^+$} and  $\textup{Inv}\,K^+_\ep$.
\end{lemma}
\begin{proof}
Integrate \eqref{eq:Pep} over $\Omega$ and sum over $1\leq i\leq N$,  we have
\begin{align*}
\frac{d}{dt}  \bigg\|\sum_{i=1}^N u_i(\cdot,t) \bigg\|_{L^1(\Omega)}  &= \bigg\|\sum_{i=1}^N m(\cdot) u_i (\cdot,t) \bigg\|_{L^1(\Omega)} - \bigg\|\sum_{i=1}^N u_i(\cdot,t) \bigg\|_{L^2(\Omega)}^2. \\
&\leq (\sup_\Omega m) \bigg\|\sum_{i=1}^N u_i(\cdot,t) \bigg\|_{L^1(\Omega)} - \frac{1}{|\Omega|} \bigg\|\sum_{i=1}^N u_i(\cdot,t) \bigg\|_{L^1(\Omega)}^2, 
\end{align*}
where we used Cauchy-Schwartz inequality for the last inequality. The assertions follow from the properties of the solution to the logistic type ODE.
\end{proof}
\begin{lemma}\label{lem:1.7}
Let $(\hat{u}_i)_{i=1}^N$ (resp. $(u_i)_{i=1}^N$) be a non-negative solution of \eqref{eq:P0} (resp. \eqref{eq:Pep}) with initial data in $\mathcal{N}$. 
There exists ${\alpha \in(\beta, 1)}$ and $C_1=C_1((\hat{d}_k)_{k=1}^{N_0},\Omega,m)$ (but otherwise independent of $N \geq N_0$ and $(d_i)_{i=1}^N$ satisfying \eqref{eq:pert} for some $\epsilon \in (0,\hat{d}_1/2)$) such that 
\begin{equation}\label{eq:lfinboundgeq1}
\sum_{i=1}^N \|\hat{u}_i\|_{C^{1+\alpha,(1+\alpha)/2}(\overline\Omega \times [1,\infty))} + \sum_{i=1}^N \|u_i\|_{C^{1+\alpha,(1+\alpha)/2}(\overline\Omega \times [1,\infty))} \leq C_1,
\end{equation}
where $C^{\alpha,\alpha/2}$ is the usual parabolic H\"{o}lder space with exponent $\alpha$ $($see, e.g. \cite[Chap. IV.1]{Lieberman}$)$.
\end{lemma}

\begin{proof}
   By Lemma \ref{lem:L1}, we have 
   $$
   \sup_{t\geq 1} \sum_{j=1}^N \|u_i\|_{L^1(\Omega \times [t-1,t+1])} \leq 4|\Omega|\sup_\Omega m.
   $$
    Since $\partial_t u_i - d_i \Delta u_i \leq m(x) u_i$, we can apply the local maximum principle \cite[Chapter VI, Theorem 7.36]{Lieberman} 
    to deduce that
    \begin{equation}\label{eq:Linf}
    \sup_{t\geq 1}\|u_i\|_{L^\infty(\overline\Omega \times [t-1/2,t+1])} \leq C\sup_{t\geq 1} \|u_i\|_{L^1(\Omega\times[t-1,t+1])}.
    \end{equation}
    It is essential that we have dropped the nonlinear terms involving $u_iu_j$ and work with the differential inequality when applying the local maximum principal for strong sub-solutions. In this way, the constant in \eqref{eq:Linf} can be chosen independently of initial data.

By applying a parabolic $L^p$ estimate to the parabolic equation $\partial_t u_i -d_i \Delta u_i = (m(x) - \sum u_j)u_i$ (which 
can now be regarded as  a linear parabolic equation of $u_i$ with $L^\infty$ bounded coefficients)
 and by the Sobolev embedding theorem, the above can be improved to
        \begin{equation}\label{eq:Lp}
    \sup_{t\geq 1}\|u_i\|_{C^{1+\alpha,(1+\alpha)/2}(\overline\Omega\times[t,t+1])}  \leq C'\sup_{t\geq 1} \|u_i\|_{L^1(\Omega\times[t-1,t+1])}.
    \end{equation}
    And the desired conclusion follows by summing $i$ from $1$ to $N$, 
$$        \sum_{j=1}^N\|u_i\|_{C^{1+\alpha,(1+\alpha)/2}(\overline\Omega \times [1,\infty))}  \leq C'\sup_{t\geq 1} \sum_{j=1}^N\|u_i\|_{L^1(\Omega\times[t-1,t+1])} \leq C''.$$
Since $d_i$ are uniformly bounded from above and below from zero, the constants $C,C',C''$ in the above estimates can be chosen to be independent of $N$ and $(d_i)_{i=1}^N$. This completes the proof.
\end{proof}

In summary, we have the following.

\begin{corollary}\label{cor:inv_K}
Fix $0<\hat{d}_1<...<\hat{d}_{N_0}$, $\ep \in (0,\hat{d}_1/2)$ and consider arbitrary
 $N \geq {N_0}$ and $(d_i)_{i=1}^N$ satisfying \eqref{eq:pert}. 
Then we have
$$
\left(\textup{Inv}\,K^+ \cup \textup{Inv}\,K^+_\ep\right) \subset \mathcal{N},
$$
where 
 {$\textup{Inv}\,K^+$} and $\textup{Inv}\,K^+_\ep$ are the invariant sets generated by \eqref{eq:P0} and \eqref{eq:Pep} respectively. Furthermore, there exists $C_0$ (dependent on $\hat{d}_k$, but independent of $\ep\in(0,\hat{d}_1/2)$, $N$ and $d_i$) such that for any solution $u$ of \eqref{eq:Pep} (resp. $\hat{u}$ of \eqref{eq:P0}) with initial data $u_0 \in \varphi(1,\mathcal{N}) \cup \varphi_\ep(1,\mathcal{N})$, we have
 \begin{equation}\label{eq:1.7a}
\sum_{i=1}^N\|u_i(\cdot,t)\|_{C(\overline\Omega)} +    \sum_{i=1}^N\|\hat{u}_i(\cdot,t)\|_{C(\overline\Omega)} \leq C_0  \quad \text{ for } t \geq 0,
\end{equation}
and 
\begin{equation}\label{eq:1.7b}
\sum_{i=1}^N\|(-\Delta)^{1/2}u_i(\cdot,t)\|_{C(\overline\Omega)} \leq C_0(1 + t^{-1/2})  \quad \text{ for } t \geq 0.
\end{equation}
\end{corollary}
\begin{proof}
Fix initial data $u_0 \in \varphi(1,\mathcal{N}) \cup \varphi_\ep(1,\mathcal{N})$. The first assertion is contained in Lemma \ref{lem:L1}.
Next,  
by Lemma \ref{lem:1.7},  
there exists $C_{1}$ depending on $(\hat{d}_k)_{k=1}^{N_0},\Omega$, and $m(x)$ but independent of $N$ and $(d_i)_{i=1}^N$ such that (here $(u_0)_j$ denotes the $j$-th components of $u_0$)
$$
\sum_{j=1}^N \|(u_0)_j\|_{C(\overline\Omega)} \leq C_1 \quad \text{ and }\quad \partial_\nu u_0 = 0 \quad \text{ on }\partial\Omega.
$$


We claim that 
$$
\sup_{0 \leq t \leq 1} \sum_{j=1}^N\|u_j(\cdot,t)\|_{C(\overline\Omega)} \leq C_1 e^{\sup_\Omega m}.
$$
Indeed, by using the differential inequality
\begin{equation}\label{eq:uii}
\begin{cases}
\partial_t u_i - d_i \Delta u_i \leq (\sup_\Omega m) u_i &\text{ in }\Omega \times [0,\infty).\\
\partial_\nu u_i =0 &\text{ on }\partial\Omega \times [0,\infty),\\
u_i(x,0) = (u_0)_i(x) &\text{ in }\Omega,
\end{cases}
\end{equation}
we can compare each $u_i$ with the super-solution $\overline{u}_i$ of \eqref{eq:uii}, given by
$$
\overline{u}_i(x,t):= e^{(\sup_\Omega m)t}\|(u_0)_i\|_{C(\overline\Omega)}  \quad\text{ in }\Omega\times[0,\infty),
$$
to deduce that
$$
\|u_i(\cdot,t)\|_{C(\overline\Omega)} \leq e^{\sup_\Omega m}\|(u_0)_i\|_{C(\overline\Omega)} \quad \text{ for }t  \in [0,1].
$$
Hence we have $$\sup_{0 \leq t \leq 1}\sum_{i=1}^N \|u_i(\cdot,t)\|_{C(\overline\Omega)} \leq  e^{\sup_\Omega m}\sum_{j=1}^N \|(u_0)_j\|_{C(\overline\Omega)}  \leq C_1 e^{\sup_\Omega m}.$$
Combining with \eqref{eq:lfinboundgeq1}, we deduce the boundedness of $\sup_{t \geq 0}\sum_{i=1}^N \|u_i(\cdot,t)\|_{C(\overline\Omega)}$. Since the proof for the boundeness of $\sup_{t \geq 0}\sum_{i=1}^N \|\hat{u}_i(\cdot,t)\|_{C(\overline\Omega)}$ is similar, we omit the proof. This establishes \eqref{eq:1.7a}.

Finally, we observe that each $u_i$ satisfies a non-autonomous linear parabolic equation with regular coefficients, so that \eqref{eq:1.7b} follows from \cite[(5.1.55) in Theorem 5.1.17]{Lunardi}.
\end{proof}


Recalling that $\varphi$ (resp. $\varphi_\ep$) is the semiflow generated by \eqref{eq:P0} (resp. \eqref{eq:Pep}), we now 
prove the main theorem of this section.
\begin{proposition}\label{prop:stepthree}
 Fix $(\hat{d}_k)_{k=1}^{N_0} \in \mathcal{D}$. For each $T>0$ and $\eta>0$, there exists $\ep_1$ such that 
 for $\ep \in (0,\ep_1)$, and  arbitrary $N \geq N_0$ and $(d_i)_{i=1}^N$ satisfying \eqref{eq:pert}, we have
 \begin{equation}
    {\sup_{u_0}} \| \mathcal{P} \varphi(t,u_0) - \mathcal{P}\varphi_\ep(t,u_0)\| < \eta\quad \text{ for } \,\,0 \leq t \leq T,
\end{equation}
where the supremum is taken over all $u_0 \in \varphi(1,\mathcal{N})\cup \varphi_\ep(1,\mathcal{N})$, $\mathcal{P}$ is the projection operator given in \eqref{eq:P} and the open set $\mathcal{N}$ is defined in \eqref{eq:N}.
\end{proposition}
\begin{proof}
Let {$\textup{Inv}\,K^+$} (resp. $\textup{Inv}\,K^+_\ep$) denote the maximal invariant set in $K^+$ of the semiflow $\varphi$ generated by \eqref{eq:P0} (resp. the semiflow $\varphi_\ep$ generated by \eqref{eq:Pep}). Let $\mathcal{N}$ be the  neighborhood of {$\textup{Inv}\,K^+$} specified by \eqref{eq:N}.

Let $(\hat{u}_i)_{i=1}^N = \varphi(t,u_0)$ and $({u}_i)_{i=1}^N = \varphi_\ep(t,u_0)$. Since $u_0 \in \varphi(1,\mathcal{N}) \cup \varphi_\ep(1,\mathcal{N})$, 
we can apply Corollary \ref{cor:inv_K} 
{so that the estimates \eqref{eq:1.7a} and \eqref{eq:1.7b} hold for some constant $C_0$ that is independent of $N$.}

We will estimate $u_i$ by the variation of constants formula. 
Recall the partition $\{I_k\}_{k=1}^{N_0}$ of $\{1,2,...,N\}$ given in \eqref{eq:pert}. 
For $i \in I_k$, we have
\begin{equation}\label{eq:I2b}
\begin{cases}
    \hat{u}_i(t) &= e^{t\hat{d}_k\Delta }(u_0)_i + \int_0^t e^{(t-s)\hat{d}_k\Delta}[\hat{u}_i(s)(m - \sum_{j=1}^N \hat{u}_j(s))]\,ds,\\
    u_i(t) &= e^{t\hat{d}_k\Delta }(u_0)_i + \int_0^t e^{(t-s)\hat{d}_k\Delta}[u_i(s)(m - \sum_{j=1}^N u_j(s))]\,ds  \\
    &\qquad + (d_i - \hat{d}_k) \int_0^t e^{(t-s)\hat{d}_k\Delta} \Delta u_i(s)\,ds 
\end{cases}
\end{equation}
Denoting $\hat{U} = \mathcal{P}\hat{u}$, $U = \mathcal{P} u$, and $W$ as follows 
$$
\hat{U}_k = \sum_{i \in I_k} \hat{u}_i,\quad U_k=\sum_{i\in I_k} u_i,\quad W_k=\sum_{i \in I_k} \frac{d_i-\hat{d}_k }{{{\ep}}} u_i, \quad \text{ for }1 \leq k \leq {N_0},
$$
and then adding \eqref{eq:I2b} over $i\in I_k$, we deduce
\begin{equation}\label{eq:I2c}
\begin{cases}
    \hat{U}_k(t) &= e^{t\hat{d}_k\Delta }(U_0)_k + \int_0^t e^{(t-s)\Delta}\big[\hat{U}_k(s)(m -\sum_{\ell=1}^{N_0} \hat{U}_\ell(s))\big]\,ds,\\
    U_k(t) &= e^{t\hat{d}_k\Delta }(U_0)_k + \int_0^t e^{(t-s)\Delta}\big[U_k(s)(m -\sum_{\ell=1}^{N_0} {U}_\ell(s))\big]\,ds  \\
    &\qquad + {{\ep}}  \int_0^t e^{(t-s)\hat{d}_k\Delta} \Delta W_k(s)\,ds 
\end{cases}
\end{equation}
Denoting $Q_k(t) = U_k(t) - \hat{U}_k(t)$, then by subtracting the 
first equation of \eqref{eq:I2c} from the second equation, we have 
\begin{align*}
Q_k(t)&= \int_0^t e^{(t-s)\hat{d}_k\Delta}\bigg[U_k(s)(m - \sum_{\ell=1}^{N_0} U_\ell(s))-\hat{U}_k(s)(m - \sum_{\ell=1}^{N_0} \hat{U}_\ell(s))\bigg] \,ds \\
& \quad + {{\ep}} \int_0^t e^{(t-s)\hat{d}_k\Delta} \Delta W_k(s)\,ds\\
&:= I_1(k) + {I_2(k)}.
\end{align*}
Now, {by \eqref{eq:1.7a},} we have
$$
\sum_{k=1}^{{N_0}}\|I_1(k)\| \leq C_1 \int_0^t \sum_{k=1}^{N_0}\|Q_k(s)\|\,ds,
$$
where we have used the uniform boundedness of trajectories in $X$. Moreover, by \eqref{eq:1.7b} of 
Corollary \ref{cor:inv_K}, we have 
\begin{align*}
\sum_{k=1}^{{N_0}}    \|I_2(k)\| &\leq\sum_{k=1}^{{N_0}}\sum_{i \in I_k} |d_i-\hat{d}_k  | \int_0^t \| e^{(t-s)\hat{d}_k\Delta} (-\Delta)^{1/2}\| \|(-\Delta)^{1/2} u_i(s)\|\,ds\\
    &\leq \sum_{i=1}^N \ep \int_0^t C_T (t-s)^{-1/2} \left\|(-\Delta)^{1/2} u_i(s)\right\|\,ds\\
    &\leq {{\ep}} C_T\int_0^t (t-s)^{-1/2} \sum_{i=1}^{N} \left\|(-\Delta)^{1/2} u_i(s)\right\|\,ds
\end{align*}
where the constant $C_T$ can be chosen to be uniform for $t \in [0,T]$. Note that we used 
$\sum_{k=1}^{N_0} \sum_{i\in I_k} = \sum_{i=1}^N$ 
and  that $\left\| e^{(t-s)\hat{d}_k\Delta} (-\Delta)^{1/2}\right\| \leq C_T (t-s)^{-1/2}$ (see \cite[Chapter 2]{Lunardi}) to derive the first inequality. 
Then,
\begin{align*}
    \sum_{k=1}^{N_0} \|Q_k(t)\| &\leq C_T\left[ \int_0^t \sum_{k=1}^{N_0} \|Q_k(s)\|\,ds + {{\ep}} \int_0^t (t-s)^{-1/2} \sum_{i=1}^N \left\|(-\Delta)^{1/2} u_i(s)\right\|\,ds \right]\\
    &\leq C'_T \left[ \int_0^t \sum_{k=1}^{N_0} \|Q_k(s)\|\,ds + {{\ep}} \int_0^t (t-s)^{-1/2} (1 + s^{-1/2})\,ds \right]\\
    &\leq  C''_T \left[\int_0^t \sum_{k=1}^{N_0} \|Q_k(s)\|\,ds +  \ep\right], 
\end{align*}
{where we used \eqref{eq:1.7b} for the second inequality.} Hence,
by the Gronwall's inequality, we have
\begin{equation}
\sup_{0 \leq t \leq T} \sum_{k=1}^{N_0}\|Q_k(t)\| \leq {{\ep}} C''_T e^{TC''_T } = {{\ep}} C'''_T.
\end{equation}
This proves Proposition \ref{prop:stepthree}.
\end{proof}

\section{Rough estimates for the perturbed semiflow}\label{sec:4}

\begin{definition}
Let $d>0$ and $\hat{h} \in L^\infty(\Omega)$, and define $\hat\mu(d,h)$ to be the principal eigenvalue of
\begin{equation}\label{eq:4.1c}
d \Delta \hat\psi  + \hat{h}(x) \hat\psi + \hat\mu \hat\psi = 0 \,\, \text{ in }\Omega,\,\,\text{ and }\,\, \partial_\nu \hat\psi = 0 \,\,\text{ on }\,\partial\Omega.
\end{equation}
\end{definition}
\begin{lemma}\label{lem:mmoonn}
Let $\partial_d  \mu_1$ be the derivative of $\mu_1$ with respect to $d$, then $\partial_d \mu_1 \geq 0$. Furthermore, if $\hat{h}$ is non-constant, then $\partial_d \mu_1 >0$ for all $d>0$.
\end{lemma}
\begin{proof}
 Indeed, if we differentiate \eqref{eq:4.1c} with respect to $d$ (and denote the derivative as $'$), we get
$$
\begin{cases}
-d \Delta \hat\psi' - \hat{h} \hat\psi' - \Delta \hat\psi = \hat\mu \hat\psi' + \hat\mu' \hat\psi &\text{ in }\Omega,\\
\partial_\nu \hat\psi' = 0 &\text{ on }\partial\Omega.
\end{cases}
$$
Multiplying the above by $\hat\psi$ and integrating by parts, we obtain
$$
0 \leq \int_\Omega |\nabla \hat\psi|^2\,dx = \hat\mu' \int_\Omega |\hat\psi|^2\,dx.
$$
Finally, note that the strict inequality follows from the fact that $\hat{h}$ is non-constant, so that $\hat\psi$ is also non-constant.
\end{proof}



The following is adapted from \cite[Lemma 3.9]{Dockery1998}.
\begin{lemma}\label{lem:3.9}
Fix $(\hat{d}_k)_{k=1}^{N_0}$. There exists $\delta>0$ such that for any $\ep \in (0, \min\{\hat{d}_1/2, (\hat{d}_2-\hat{d}_1)/2\})$, and any $u_0 \in \textup{Int}\,K^+$, the omega limit set $\omega(u_0,\varphi_\ep)$ of $u_0$ under the semiflow of \eqref{eq:Pep} satisfies
$$
\mathcal{P}\omega(u_0,\varphi_\ep) \not\subset B_{\delta}(M(k_0)) \quad \text{ for any }k_0 \in \{2,...,{N_0}+1\}.
$$
\end{lemma}
\begin{proof}
Suppose to the contrary that there is $k_0 \in \{2,...,{N_0}+1\}$ such that for each $\delta>0$, there is $T_0\geq 1$ such that the solution $(u_i)_{i=1}^N=\varphi_\ep(t, u_0)$ satisfies
\begin{equation}\label{eq:uu?}
\mathcal{P}\varphi_\ep([T_0,\infty), u_0) \subset B_\delta(M(k_0)).
\end{equation}
Define, $h(x,t):=m(x) - \sum_{j=1}^N u_j(x,t)$ and  
$$
h_\delta(x):= \inf_{(\hat{v}_k) \in B_\delta(M(k_0))} (m(x) - \sum_{k=1}^{N_0} \hat{v}_k(x)), 
$$
then $h(x,t) \geq h_\delta(x)$ for all $t \geq T_0$. We claim that
$\mu_1(d_1,h_\delta) <0$ for all sufficiently small $\delta$. 

We first discuss the case $2\leq k_0 < {N_0}+1$. By continuity, it suffices to show that $\mu_1(d_1,m-\theta_{\hat{d}_{k_0}}) <0$. Since $d_1 < \hat{d}_{k_0}$ (as $d_1 < \hat{d}_1 + \ep < \hat{d}_2 \leq  \hat{d}_{k_0}$), we can apply the classical fact that $\mu_1$ is strictly increasing in $d$ (Lemma \ref{lem:mmoonn}), to deduce that
$$
\mu_1(d_1,m-\theta_{\hat{d}_{k_0}}) < \mu_1(d_{\hat{d}_{k_0}},m-\theta_{\hat{d}_{k_0}}) = 0, 
$$
where the last equality holds since $0$ is the eigenvalue with a positive eigenfunction $\theta_{\hat{d}_{k_0}}(x)$. 

In case $k_0 = {N_0}+1$, it suffices to show that $\mu_1(d_1,m) < 0$. Indeed, if $\psi >0$ is the principal eigenfunction of $\mu_1(d_1,m)$, then
$$
d_1 \Delta \psi + m\psi + \mu_1(d_1,m) \psi = 0 \,\,\text{ in }\Omega,\, \text{ and }\, \partial_\nu \psi = 0 \, \text{ on }\partial\Omega.
$$
Dividing the above by $\psi$ and integrate by parts, we have
\begin{equation}\label{eq:hhhold}
d_1 \int_\Omega \frac{|\nabla \psi|^2}{\psi^2} \,dx + \int_\Omega m \,dx + \mu_1(d_1,m) |\Omega| = 0
\end{equation}
Since (i) $\int_\Omega m\,dx \geq 0$ and that (ii) $m(x)$, and thus $\psi(x)$, is non-constant, we deduce from \eqref{eq:hhhold} that $\mu_1(d_1,m)<0$.

In conclusion, there exists $\delta>0$ such that the principal eigenvalue $\lambda$ of
$$
d_1 \Delta \psi  + h_\delta(x) \psi + \lambda \psi = 0 \,\, \text{ in }\Omega,\,\,\text{ and }\,\, \partial_\nu \psi = 0\,\,\text{ on }\,\partial\Omega.
$$
is negative. 
We also normalize the corresponding positive eigenfunction $\psi$ to insure $\inf_\Omega \psi = 1$. 
Now, since $h(x,t)\geq h_\delta(x)$, we can show that
$u_1(x,t)$ and
$$
 \underline{u}_1(x,t):= 2\delta e^{-\lambda (t-t_0)}\psi(x),
$$
where $t_0 \in [T_0,\infty)$ is to be specified,
together form a pair of super- and sub-solutions of the linear parabolic equation
$$
\partial_t w - d_1 \Delta w = h(x,t) w  \quad \text{ in }\Omega \times [T_0,\infty),
$$
under the Neumann boundary condition. By taking $t_0 \geq T_0$ sufficiently large, we have also $u_1(x,T_0) \geq \underline{u}_1(x,T_0)$. By the method of sub- and super-solutions, $u_1(x,t) \geq \underline{u}_1(x,t)$ for all $t \geq T_0$. However, when $t=t_0$, we have 
$$
u_1(x,t_0) \geq 2\delta \psi(x) \geq 2\delta \quad \text{ for all }x \in \Omega,
$$
but this contradicts \eqref{eq:uu?} when $t=t_0$.
\end{proof}

{The next result is inspired by the proof of \cite[Theorem 4.2]{Dockery1998}.}
\begin{proposition}\label{prop:1.14}
Given $(\hat{d}_k)_{k=1}^{N_0} \in \mathcal{D}$ and a sufficiently small $r>0$, then there exists $\ep>0$ such that for any $N \geq N_0$ and $(d_i)_{i=1}^N$ such that \eqref{eq:pert} holds, and any {positive} solution $u$ of \eqref{eq:Pep}, we have 
\begin{equation}\label{eq:cor2.1}
\limsup_{t \to \infty}\left\| \sum_{i =1}^N u_i(\cdot,t) - \theta_{\hat{d}_1} \right\|< 2r,
\end{equation}
\end{proposition}
\begin{proof}
It suffices to prove \eqref{eq:cor2.1} for sufficiently small $r>0$. 
Let $\varphi$ (resp. $\varphi_\ep$) be the semiflow operator corresponding to \eqref{eq:P0} (resp. \eqref{eq:Pep}), and denote its maximal bounded invariant set to be 
{$\textup{Inv}\,K^+$} (resp. $\textup{Inv}\,K^+_\ep$). 
Let $V$ be the Lyapunov function that is given in Section \ref{sec:2}.
Fix $\mu=3/4$ and  $r \in (0, \min\{1/4,\delta/2\})$ small enough (with $\delta$ given by Lemma \ref{lem:3.9}) so that 
\begin{equation}\label{eq:ball}
|V(U) - V(\tilde{U})| < \frac{1}{4}  \quad \text{ if }U,\tilde{U} \in B_{2r}(M(k))\,\text{ for some }1 \leq k\leq {N_0}+1.
\end{equation}
Since $V(M(k))=k$, it follows that 
\begin{equation}\label{eq:ball2}
V(B_{2r}(M(k))) \subset \left( k - \frac{1}{4}, k+\frac{1}{4}\right) \quad \text{ for each }k.
\end{equation}
Having chosen $\mu$ and $r$, we then choose
$T>1$ and $\mathcal{U}$ so that the conclusion of 
Lemma \ref{lem:1.5}
holds. (Since $\textup{Inv}\,K^+ \cup \textup{Inv}\,K^+_\ep \subset \mathcal{N}$, we can also assume that $\mathcal{U} \subset \mathcal{N}.$)
\begin{claim}\label{claim1}
There exists $\ep_1 \in (0,\hat{d}_1/2)$ such that for all $\ep \in (0,\ep_1)$ and any $N\geq N_0$ and $(d_i)_{i=1}^N$ satisfying \eqref{eq:pert}, we have
$$
(\mathcal{P}\textup{Inv}\,K^+) \cup (\mathcal{P}\textup{Inv}\,K^+_\ep) \subset \mathcal{U}.
$$
\end{claim}
{Let ${F}_\ep:= \varphi_\ep(1,\mathcal{N})$. 
By Lemma \ref{lem:1.7}, we can choose a compact set ${F}_0$ in $[C(\overline\Omega)]^{N_0}$ such that
$\mathcal{P} F_\ep \subset F_0$.
Since $\mathcal{U}$ is a neighborhood of the maximal bounded invariant set of the ${N_0}$-species problem \eqref{eq:p0}, there exists a finite time $T_0 \ge 1$ such that $\Phi(T_0, F_0) \subset \mathcal{U}$. We emphasize that $F_0$ and $T_0$ can be chosen uniformly over all $N \geq N_0$ and $(d_i)_{i=1}^N$ satisfying \eqref{eq:pert}. }

{By Remark \ref{rmk:1} and $\mathcal{P} F_\ep \subset F_0$, we have 
$$
\mathcal{P}\varphi(T_0,F_\ep) = \Phi(T_0,\mathcal{P}F_\ep) \subset \Phi(T_0,F_0) \subset \mathcal{U}.
$$
Also, since $\Phi(T_0,F_0)$ is compact (as $F_0$ is compact) and $\mathcal{U}$ is open, there exists $\eta>0$ (again independent of $N$ and $(d_i)_{i=1}^N$) such that $$\textup{dist}(\Phi(T_0,F_0), [C(\overline\Omega)]^{N_0} \setminus \mathcal{U})\geq \eta.$$ We can then apply Proposition \ref{prop:stepthree} to choose $\ep_1>0$ so small  that for any
{$\ep\in(0,\ep_1)$,} $N  \geq N_0$ and $(d_i)_{i=1}^N$ satisfying \eqref{eq:pert},  we have
$\mathcal{P}\varphi_\ep(T_0,F_\ep)  \subset \mathcal{U}.$
And hence, 
$$\mathcal{P}\textup{Inv}\,K^+_\ep \subset  \mathcal{P}\varphi_\ep(T_0+1,\mathcal{N}) =\mathcal{P}\varphi_\ep(T_0,F_\ep)   \subset \mathcal{U},$$
where the first inclusion is due to the last part of Lemma \ref{lem:L1}, while the equality is due to $F_\ep = \varphi_\ep(1,\mathcal{N})$ and semigroup property. This proves the claim.}

Next, by Proposition \ref{prop:stepthree}, there exists $\ep_2\in (0,\ep_1)$ such that for all $\ep \in (0,\ep_2)$ and $u'_0 \in \varphi(1,\mathcal{N})\cup \varphi_\ep(1,\mathcal{N})$, 
we have 
(recall the choice of $T$ in the beginning of the proof)
\begin{equation}\label{eq:1.21a}
\sup_{0 \leq t \leq 2T}\left\| \mathcal{P}
    \varphi(t,u'_0) - \mathcal{P}\varphi_\ep(t,u'_0)\right\| < r,
\end{equation}
and, provided $\mathcal{P}\varphi([0,2T],u'_0) \subset \mathcal{U}$ and $\mathcal{P}\varphi_\ep([0,2T],u'_0) \subset \mathcal{U}$, that 
\begin{equation}\label{eq:1.21}
\sup_{0 \leq t \leq 2T}\left\| V(\mathcal{P}
    \varphi(t,u'_0)) - V(\mathcal{P}\varphi_\ep(t,u'_0))\right\| < \frac{\mu}{3}.
\end{equation}

Now, fix an arbitrary trajectory $\varphi_\ep(t,u_0)$ of \eqref{eq:Pep} with initial data $u_0$ {being an interior point of $K^+$.}
%
We will show \eqref{eq:cor2.1}. By Claim \ref{claim1}, we may perform a translation in time and assume without loss of generality that
$$
\mathcal{P}\varphi(t,u_0) \in \mathcal{U}, \quad \text{ and }\quad \mathcal{P}\varphi_\ep(t,u_0) \in \mathcal{U} \quad \text{ for all }t \geq 0.
$$
\begin{claim}\label{claim:1correct}
Let $\ep \in (0,\ep_2)$. 
Suppose there is $t_1\geq 1$ such that 
\begin{equation}\label{eq:1.26b}
V(\mathcal{P}\varphi_\ep(t_1,u_0))=k_1 + \frac{1}{2}, \quad \text{ for some }\, k_1\in \{1,...,{N_0}+1\},
\end{equation}
then 
\begin{equation}\label{eq:claim1correct1}
\sup_{t_1<t<t_1+T} V(\mathcal{P}\varphi_\ep(t,u_0))< k_1 + \frac{3}{4},
\end{equation}
and
\begin{equation}\label{eq:claim1correct12}
V(\mathcal{P}\varphi_\ep(t_2,u_0)) \leq k_1 + \frac{1}{2} \quad \text{ for some }t_2 \in (t_1,t_1+T].
\end{equation}
\end{claim}
Denote $u'_0= \varphi_\ep(t_1,u_0)$. Then \eqref{eq:ball} and \eqref{eq:1.26b} imply $\mathcal{P}u'_0 \not\in \cup_{k=1}^{{N_0}+1}B_r(M(k))$. Then, we have
$$
V(\mathcal{P}\varphi_\ep(t,u'_0)) <V(\mathcal{P}\varphi(t,u'_0)) + \frac{\mu}{3}\leq V(\mathcal{P}u'_0) + \frac{\mu}{3} = k_1 + \frac{3}{4} \quad \text{ for }t \in [0,T],
$$
where the first inequality is due to \eqref{eq:1.21} and the second one is due to $\mathcal{P}\varphi(t,u'_0) = \Phi(t,\mathcal{P}u'_0)$ and the property of Lyapunov function. 
This proves \eqref{eq:claim1correct1}. 

Next, we show \eqref{eq:claim1correct12} by dividing into two cases: \begin{itemize}
    \item[{\rm(i)}] $\mathcal{P}\varphi([0,T],u'_0) \cap \left[ \cup_{k=1}^{{N_0}+1}B_r(M(k))\right] = \emptyset$
    \item[{\rm(ii)}] $\mathcal{P}\varphi([0,T],u'_0) \cap \left[ \cup_{k=1}^{{N_0}+1}B_r(M(k))\right] \neq \emptyset$
\end{itemize}
In case (i), we use \eqref{eq:1.21} and then 
Lemma \ref{lem:1.5}
to obtain
$$
V(\mathcal{P}\varphi_\ep(T,u'_0)) < V(\mathcal{P}\varphi(T,u'_0)) + \frac{\mu}{3} < V(\mathcal{P}u'_0) - \frac{2\mu}{3} = k_1.  
$$
In case (ii), 
there is an integer $k'_1$ and $t'_1 \in (0,T]$ such that $\mathcal{P}\varphi(t'_1,u'_0) \in B_r(M(k'_1))$. Furthermore, since $V$ is decreasing along trajectories of $\Phi = \mathcal{P}\varphi$, we have $k'_1 \leq k_1$.
Then,  
$$
V(\mathcal{P}\varphi_\ep(t'_1,u'_0)) < V(\mathcal{P}\varphi(t'_1,u'_0)) + \frac{\mu}{3} < k'_1 + \frac{1}{4} + \frac{\mu}{3} \leq k_1 + \frac{1}{2},
$$
where the first and second inequalities are due to \eqref{eq:1.21} and \eqref{eq:ball}, respectively. 
This proves \eqref{eq:claim1correct12} and completes the proof of Claim \ref{claim:1correct}.

\begin{claim}\label{claim:2correct}
For $\ep\in(0,\ep_2)$, there exists $k_0$ and $T_\ep$ such that
\begin{equation}\label{eq:claim:2correct}
k_0-\frac{1}{2} < V(\mathcal{P} \varphi_\ep(t,u_0))< k_0 + \frac{3}{4} \quad \text{ for all }t \geq T_\ep.
\end{equation}
\end{claim}
\begin{remark}
It follows from \eqref{eq:ball2} and \eqref{eq:claim:2correct} 
that for $t \geq T_\ep$, we have $\mathcal{P}\varphi_\ep(t,u_0) \not\in B_{2r}(M(k))$ for any $k \neq k_0$.
\end{remark}

Define 
$$
k_0 = \min\left\{ k \in \mathbb{N}\,:\, V(\mathcal{P}\varphi_\ep(t_0,u_0) )\leq k+\frac{1}{2}\,\text{ for some }t_0 \geq 1\right\}.$$ 
By construction, $V(\mathcal{P} \varphi_\ep(t,u_0)) > (k_0-1) + 1/2$ for all $t \geq 1$ and
the lower bound of \eqref{eq:claim:2correct} holds. Moreover, there is $T_\ep\geq 1$ such that $V(\mathcal{P}\varphi_\ep(T_\ep,u_0)) \leq k_0 + \frac{1}{2}$. Denote, for simplicity, $u'_0=\varphi_\ep(T_\ep,u_0)$. Suppose to the contrary that 
$$
V(\mathcal{P}\varphi_\ep(t_3, u'_0) \geq k_0 + \frac{3}{4} \quad \text{ for some }t_3 > 0. 
$$
Define the set 
$$
S= \left\{t \in [0,t_3):\,V(\mathcal{P}\varphi_\ep(t_3, u'_0)) \leq k_0 + \frac{1}{2}\right\},
$$
then $S$ is non-empty since $0 \in S$. 
Let $t_4 = \sup S$. By \eqref{eq:claim1correct1} we have $t_3 - t_4 > T$. By \eqref{eq:claim1correct12}, we have $t_5 \in (t_4,t_3)$ such that $t_5 \in S$. This contradicts the definition of $t_4$. This proves \eqref{eq:claim:2correct}.


\begin{claim}\label{claim:3correct}
There exists $\ep_3 \in (0,\ep_2)$ such that if $\ep \in (0,\ep_3)$, for any $u_0 \in \textup{Int}\,K^+$ with certain $k_0$ guaranteed by Claim \ref{claim:2correct}, it follows that
\begin{equation}\label{eq:uuu}
\mathcal{P}\omega(u_0,\varphi_\ep) \subset B_{2r}(M(k_0)).
\end{equation}
\end{claim}
Suppose to the contrary that there is a certain $1\leq k_0 \leq {N_0}+1 $, a sequence $\ep \to 0$ and $N = N^\ep$ and $(d^\ep_i)_{i=1}^{N^\ep}$ and $u_0=u_0^\ep$ such that the conclusions of Claim \ref{claim:2correct} hold for that $k_0$ {but \eqref{eq:uuu} is false}. Let $v_\ep = \varphi_\ep(T_\ep,u^\ep_0)$. {Then $\mathcal{P}\omega(v_\ep,\varphi_\ep) \not\subset B_{2r}(M(k_0))$.} 
Let $w_\ep \in \omega(v_\ep,\varphi_\ep)$ such that $\mathcal{P}w_\ep  \not\in B_{2r}(M(k_0))$.

Thanks to the {\it a priori} estimates developed in Lemma \ref{lem:1.7},  $\{\mathcal{P}w_\ep\}$ belongs to a compact set. Therefore,  
we can pass to a sequence $\ep \to 0$ such that $\mathcal{P}w_\ep \to \hat{W}$. Taking \eqref{eq:claim:2correct} into account, 
$$
k_0 - \frac{1}{2} \leq V(\mathcal{P} \varphi_\ep(t,w_\ep)) \leq k_0 + \frac{3}{4} \quad \text{ for all }t \in \mathbb{R}.
$$
As a result, we have $\hat{W} \in \textup{Inv}\,K^+$ and $k_0 - \frac{1}{2} \leq  V(\Phi(\mathbb{R},\hat{W})) \leq k_0 + \frac{3}{4}$, where we implicitly used the observation in Remark \ref{rmk:1}. However, since $M(k_0)$ is the maximal invariant set in 
$$
\left\{ W \in \text{Inv}\,K^+:\, k_0 - \frac{1}{2} \leq V(W) \leq k_0 + \frac{3}{4}   \right\} ,
$$
we are led to the conclusion that
$\{\hat W\} = M(k_0)$. This is a contradiction since $\hat{W} \not\in B_{2r}(M(k_0))$. The proves Claim \ref{claim:3correct}.

Now, in view of Lemma \ref{lem:3.9}, $\mathcal{P}\omega(u_0,\varphi_\ep) \not\subset B_{2r}(M(k))$ for all $k >1$. Hence, for any $u_0 \in \text{Int}\,K^+$, \eqref{eq:uuu} holds with $k_0 = 1$. Since $M(1) = \{(\theta_{\hat{d}_1},0,...,0)\}$, this means 
$$
\limsup_{t\to\infty} \big\| \mathcal{P} \varphi_\ep(t,u_0) - (\theta_{\hat{d}_1},0,...,0)  \big\|<2r.
$$
This proves  \eqref{eq:cor2.1} and completes the proof of the proposition. 
\end{proof}

\section{The normalized principal Floquet bundle}\label{sec:5}
In this section, we {recall} the notion of a normalized principal Floquet bundle {(see \cite{Polacikbd})}, which is a generalization of the notion of principal eigenfunction of an elliptic, or periodic-parabolic operator. We give a theorem concerning its smooth dependence on parameters.
This property will be crucial in
the proof of Theorem \ref{thm:main}. 

\subsection{The normalized principal Floquet bundle}

Given $d>0$ and $h(x,t) \in C^{\beta,\beta/2}(\overline\Omega \times \mathbb{R})$, we say that the pair $(\psi_1(x,t), H_1(t))$ is the corresponding normalized principal Floquet bundle if it satisfies
\begin{equation}\label{eq:bundlen'}
\left\{
\begin{array}{ll}
\partial_t \psi_1(x,t) -d \Delta \psi_1(x,t) - h(x,t)\psi_1(x,t) =  H_1(t)\psi_1(x,t)  &\text{ for }x \in \Omega,\, t \in \mathbb{R},\\
\partial_\nu \psi_1(x,t) = 0 &\text{ for }x \in \partial \Omega, \, t \in \mathbb{R},\\
{\int_\Omega \psi_1(x,t)\,dx = 1} & \text{ for  } t\in \mathbb{R},\\
\psi_1(x,t) >0   &\text{ for }x \in \Omega,\, t \in \mathbb{R}.
\end{array}
\right.
\end{equation}
The existence and uniqueness of $(\psi_1(x,t), H_1(t))$ is {recalled} in 
Theorem \ref{thm:A1}. 

\begin{remark}\label{rmk:2}If $h(x,t) = \hat{h}(x)$ for some time-independent function $\hat{h}$, then $\psi_1$ and $H_1$ are time-independent, and coincide with  the principal eigenfunction and principal eigenvalue $(\hat\psi(x),\hat\mu)$ of 
\begin{equation}\label{eq:pev}
-d\Delta \hat\psi - \hat{h}(x) \hat\psi = \hat\mu \hat\psi\quad \text{ in }\Omega,\quad \partial_\nu \hat\psi = 0\quad \text{ on }\partial\Omega.
\end{equation}
\end{remark}

The main result of this section is the smooth dependence of the principal Floquet bundle on the coefficients.  
\begin{proposition}\label{prop:A2'}
The normalized principal Floquet bundle, as a mapping
$$
\begin{array}{rl}
(d,h) & \mapsto (\psi_1,H_1)\\
\mathbb{R}_+ \times C^{\beta,\beta/2}(\overline\Omega\times \mathbb{R}) & \to C^{2+\beta,1+\beta/2}(\overline\Omega\times \mathbb{R}) \times  C^{\beta/2}(\mathbb{R})
\end{array}
$$
is smooth.
\end{proposition}
Since the proof of Proposition \ref{prop:A2'} is self-contained, we will postpone it to the Appendix. See Proposition \ref{prop:A2} for details. Next, let us recall that $0<\beta<1$ is fixed throughout the paper by \eqref{eq:MmM}.

\begin{corollary}\label{cor:A2'}
Let $\delta>0$, and $\hat{h}(x) \in C^\beta(\overline\Omega)$ be a non-constant function that depends on $x$ only. There exists $0<r'<1$ such that for any $d>0$ and any function $h(x,t) \in C^{\beta,\beta/2}(\overline\Omega \times \mathbb{R})$, if
\begin{equation}\label{eq:small}
\delta < d < 1/\delta,\quad 
\|h - \hat{h}\|_{C^{\beta,\beta/2}(\overline\Omega \times \mathbb{R})} < r',
\end{equation}
then the partial derivative $\partial_d H_1(t)$ of $H_1(t)$, with respect to the diffusion coefficient $d$, satisfies
$$
\inf_{t \in \mathbb{R}} \partial_d H_1(t) \geq r'.
$$
\end{corollary}
\begin{proof}
Denote by $(\psi_1(x,t;d,h), H_1(t;d,h))$ the normalized principal Floquet bundle satisfying \eqref{eq:bundlen'} for some constant $d>0$ and function $h$.  
By Remark \ref{rmk:2}, we see that 
$$
(\psi_1(x,t;d,\hat{h}),H_1(t;d,\hat{h}) = (\hat\psi(x) , \hat\mu)
$$
where $(\hat\psi(x),\hat\mu)$ is the principal eigenpair of \eqref{eq:pev}.

By Lemma \ref{lem:mmoonn}, $\partial_d\hat\mu >0$ for all $d>0$. 
In particular, 
$$
r_0:= \inf\limits_{\delta\leq d \leq 1/\delta} \partial_d\hat\mu >0
$$
Now it follows from Proposition \ref{prop:A2'} that there exists $r'\in (0,r_0/2)$ such that if \eqref{eq:small} holds, then
$$
\|\partial_d H_1(\cdot; d,h) - \partial_d H_1(\cdot; d, \hat h)\|_{C^{\beta,\beta/2}(\overline\Omega \times \mathbb{R})} = \|\partial_d H_1(\cdot; d,h) - \partial_d \hat\mu\|_{C^{\beta,\beta/2}(\overline\Omega \times \mathbb{R})}  < \frac{r_0}{2}.
$$
Hence
$$
\inf_{t \in \mathbb{R}} \partial_d H_1(t;d,h) > \partial_d \hat\mu - \frac{r_0}{2} \geq \frac{r_0}{2} \quad \text{ for }\delta < d < 1/\delta.
$$
This proves the corollary.
\end{proof}

\section{Completion of the proof of main theorem}\label{sec:6}
Recall that $m \in C^{\beta+\epsilon'}(\overline\Omega)$ where $0<\beta <1$ and $0<\epsilon'<1-\beta$.
\begin{proposition}\label{prop:2.1}
Given $\delta>0$, and a non-constant function {$\hat{h}(x) \in C^{\beta+\epsilon'}(\overline\Omega)$} that depends on $x$ only, there exists $r_1>0$ such that for any $N \in \mathbb{N}$ and $(d_i)_{i=1}^N$ satisfying
$$
\delta < d_1<...<d_N < \frac{1}{\delta},
$$
if a positive solution $u$ of \eqref{eq:1.1} satisfies 
\begin{equation}\label{eq:prop2.1}
\limsup_{t \to \infty}\Big\|\Big(m- \sum_{i =1}^N u_i(\cdot,t)\Big) - \hat{h} \Big\|_{C(\overline\Omega)}< r_1,
\end{equation}
then $u \to (\theta_{d_1},0,...,0)$ as $t\to\infty$.
\end{proposition}
Before we give the proof of Proposition \ref{prop:2.1}, an immediate consequence can be stated as follows:
If $$
\mathcal{P} \omega(u_0,\varphi_\ep) \subset B_{r_1}((\theta_{\hat{d}_1},0,\dots,0))
$$
then $\omega(u_0,\varphi_\ep) = \{(\theta_{d_1},0,\dots,0)\}.$ Compare this with Lemma \ref{lem:3.9'} of this article, which is proved in Dockery et al. \cite{Dockery1998}.

\begin{corollary}\label{cor:2.1}
Given constants $0 <\hat{d}_1<...<\hat{d}_{N_0}$ such that $(\hat{d}_k)_{k=1}^{N_0} \in \mathcal{D}$, there exists $r_2>0$ such that for any  $N \in \mathbb{N}$ and $(d_i)_{i=1}^N$ satisfying
$$
0 < d_1<...<d_N\quad \text{ and }\quad \textup{dist}_H((d_i)_{i=1}^N, (\hat{d}_k)_{k=1}^{N_0}) < \frac{1}{2}\hat{d}_{1},
$$
if a positive solution $u$ of \eqref{eq:Pep} satisfies \eqref{eq:cor2.1} {with $r =r_2$},
then $u \to (\theta_{d_1},0,...,0)$ as $t\to\infty$.
\end{corollary}
\begin{proof}[Proof of Corollary \ref{cor:2.1}]
To apply Proposition \ref{prop:2.1}, it suffices to check that, for each $\hat d>0$, the function $\hat{h}(x)=m(x) - \theta_{\hat{d}}(x)$ is non-constant in $x$. Suppose to the contrary that, 
$m(x) -\theta_{\hat{d}}(x) = \lambda$ for some constant $\lambda$. Then 
$$
\hat{d} \Delta \theta_{\hat{d}} + \lambda \theta_{\hat{d}} = 0\,\,  \text{ in }\Omega,\,\, \text{ and }\,\, \partial_\nu \theta_{\hat{d}} = 0 \,\,\text{ on }\partial\Omega.
$$
i.e. $\lambda/\hat d$ is an eigenvalue of the Laplacian operator in the domain $\Omega$ subject to the Neumann boundary condition. Since $\theta_{\hat{d}}$ constitutes a positive eigenfunction, it must be the case that $\lambda =0$ and $\theta_{\hat{d}} = C$ for some constant $C$. However, this implies that $m(x) = \theta_{\hat{d}} = C$ as well. This is a contradiction to the standing assumption that $m(x)$ is a non-constant function.
\end{proof}

\begin{proof}[Proof of Proposition \ref{prop:2.1}]
Given $\delta>0$ {and $\hat{h}(x) \in C^{\beta+\epsilon'}(\overline\Omega)$,} let $0<r'<1$ be as given in Corollary \ref{cor:A2'}. {We claim that there is $0<r_1<1$, depending on $r'$, such that if a positive solution $u$ of \eqref{eq:1.1} satisfies \eqref{eq:prop2.1} with such an $r_1$, then}
%
\begin{equation}\label{eq:iii}
\limsup_{t \to \infty}\Big\|\Big(m - \sum_{i =1}^N u_i(\cdot,t)\Big) - \hat{h} \Big\|_{C^{\beta,{\beta/2}}(\overline\Omega\times[t,t+1])}< r'.
\end{equation}
Indeed, the fact that $m(x),\hat{h}(x) \in C^{\beta+\epsilon'}(\bar\Omega)$ and the  {\it a priori} estimate \eqref{eq:lfinboundgeq1} in Lemma \ref{lem:1.7} together imply that
$$
\Big\|\Big(m - \sum_{i =1}^N u_i\Big) - \hat{h} \Big\|_{C^{\beta+\epsilon',(\beta+\epsilon')/2}(\overline\Omega \times [1,\infty))} \leq C.
$$
{By interpolating with \eqref{eq:prop2.1}}, we can show that 
\begin{align*}
&\Big\|\Big(m - \sum_{i =1}^N u_i\Big) - \hat{h} \Big\|_{C^{\beta,\beta/2}(\overline\Omega\times[t,t+1])} \\
&\leq C \Big\|\Big(m(x)- \sum_{i =1}^N u_i\Big) - \hat{h} \Big\|^\gamma_{C(\overline\Omega\times[t,t+1])}\leq C r_1^\gamma \quad \text{ for }t \gg 1,
\end{align*}
where $C$ and $\gamma = \epsilon'/(\beta + \epsilon')$ are positive constants in the interpolation inequality. Hence, we deduce \eqref{eq:iii} upon taking $r_1\in (0, (r'/C)^{1/\gamma}]$. 

Next, define $h(x,t) = m(x) - \sum_{j=1}^N u_j(x,t)$. After an appropriate translation in time, we may assume without loss of generality that
\begin{equation}\label{eq:1.deltaclose}
\|h(\cdot,t)-\hat{h}(\cdot)\|_{{C^{\beta,\beta/2}}(\overline\Omega \times [0,\infty))} <r'.
\end{equation}
Extend $h(x,t)$ evenly in $t$, so that it is defined for $(x,t) \in \Omega \times \mathbb{R}$. Let $\psi_1(x,t;d,h)$ and $H_1(t;d,h)$ be the normalized principal Floquet bundle guaranteed by Section \ref{sec:5}.   By an application of Corollary \ref{cor:A2'}, we have for any $d \in [\delta,1/\delta]$,
\begin{equation}\label{eq:spectralgap}
\inf_{t \in \mathbb{R}} \partial_dH_1(t;d,h) \geq r' >0.
\end{equation}

For each $i$, we claim that there is $\overline{c}_i > \underline{c}_i >0$ such that 
\begin{equation}\label{eq:spectraldeomp}
\underline{c}_i e^{-\int_0^t H_1(s;d_i,h)\,ds} \psi_1(x,t;d_i,h) \leq     u_i(x,t) \leq  \overline{c}_i e^{-\int_0^t H_1(s;d_i,h)\,ds}  \psi_1(x,t;d_i,h) 
\end{equation} for {$(x,t)\in\Omega\times [0,\infty)$.}

Indeed, the left and right hand sides of \eqref{eq:spectraldeomp} satisfy the same equation as $u_i$. Hence we can choose $\overline{c}_i$ large enough and $\underline{c}_i$ small enough to deduce \eqref{eq:spectraldeomp} from classical comparison theorem of linear parabolic equations in the domain $\Omega \times [0,\infty)$. This proves \eqref{eq:spectraldeomp}.

By \eqref{eq:spectralgap}, we have
$$
H_1(t;d_{i},h)-H_1(t;d_1,h) \geq (d_{i}-d_1)r' >0\quad \text{ for all }2\leq i\leq N,\, t \in \mathbb{R}.
$$
{Next, we claim  
\begin{equation}\label{eq:hhar}
\frac{1}{C_\delta} \leq \psi_1(x,t;d,h) \leq C_\delta \quad \text{ in }\Omega \times \mathbb{R},\, d\in [\delta,1/\delta].
\end{equation}
Indeed, each $\psi_1$ satisfies a linear heat equation with Neumann boundary condition and $L^\infty$ bounded coefficients. By \cite[Theorem 2.5]{Huska2006}, there is a constant C depending on $\delta,\|h\|_{L^\infty(\Omega \times \mathbb{R})}$ but independent of $t \in \mathbb{R}$ and $d \in [\delta,1/\delta]$ such that
$$
\sup_{x \in \Omega} \psi_1(x,t;d,h) \leq C\inf_{x \in \Omega} \psi_1(x,t;d,h).
$$
Combining with the normalization $\int_\Omega \psi_1(x,t;d,h) \,dx = 1$, we deduce
the estimate \eqref{eq:hhar}.}

Using \eqref{eq:spectralgap} and \eqref{eq:hhar}, we derive from \eqref{eq:spectraldeomp} that, for $i>1$,
\begin{align*}
\frac{u_{i}(x,t)}{u_1(x,t)} &\leq C \exp\left(-\int_0^t(H_1(s;d_{i})-H_1(s;d_1))\,ds \right)  \frac{\psi_1(x,t;d_{i},h)}{\psi_1(x,t;d_{1},h) }\\ 
&\leq C\exp\left(-(d_{i}-d_1)r't\right) \to 0 \quad \text{ as }t\to\infty.
\end{align*}

    Since we also have $\limsup\limits_{t \to \infty} \|u_1(\cdot,t)\| \leq C_1$ (by Lemma \ref{lem:1.7}), we deduce that $u_i \to 0$ uniformly for $i=2,..,N$. Hence the semiflow $\varphi_\ep$ generated by \eqref{eq:Pep} is asymptotic to the single species model consisting of only the first species $u_1$. Since the trivial solution is repelling by our assumptions $\int_\Omega m\,dx \geq 0$ and $m$ is non-constant (see Lemma \ref{lem:3.9}), we deduce that $u_1 \to \theta_{d_1}$ uniformly as $t\to\infty$.
\end{proof}

Recall that a subset $A$ of $K^+$ is said to be {\it internally chain transitive} with respect to the semiflow $\varphi$ if, for two points $u_0,v_0 \in A$, and any $\delta>0$, $T>0$, there is a finite sequence $$
\mathcal{C}_{\delta,T}=\{u^{(1)}=u_0,u^{(2)},...,u^{(m)}=v_0; t_1,...,t_{m-1} \}
$$
    with $u^{(j)} \in A$ and $t_j \geq T$, such that $\|\varphi(t_j,u^{(j)}) - u^{(j+1)}\| < \delta$ for all $1 \leq i \leq m-1$. The sequence $\mathcal{C}_{\delta,T}$ is called a $(\delta,T)$-chain connecting $u_0$ and $v_0$. 
    
    Let $E,E'$ be two equilibrium points.  $E$ is said to be chained to $E'$, written as $E \to E'$, if there exists a full trajectory $\varphi(t,u_0)$ (though some $u_0$ distinct from $E,E'$) such that $\alpha(u_0,\varphi)=E$ and $\omega(u_0,\varphi)=E'$. We recall that $E_i$ denotes the equilibrium $(0,...,0,\theta_{d_i},0,...,0)$ of the semiflow $\varphi_\ep$.

\begin{proof}[Proof of Theorem \ref{thm:main}]
Let $\{\hat{d}_k\}_{k=1}^{N_0} \in \mathcal{D}$. Then let {$r_2>0$} be given by 
Corollary \ref{cor:2.1}. By Proposition \ref{prop:1.14}, there exists $\overline\ep \in (0,\hat{d}_1/2)$ such that for any $N\in \mathbb{N}$ and any choice of diffusion rates $(d_i)_{i=1}^N$ such that $\textup{dist}_H((d_i),(\hat{d}_k)) < \overline\ep$, the estimate
\eqref{eq:cor2.1} {with $r=r_2$} holds for all positive solutions. Hence by Corollary \ref{cor:2.1}, the equilibrium $E_1$ attracts all solutions of \eqref{eq:1.1} with initial data in $\textup{Int}\,K^+$.

It remains to verify that  $(d_i)_{i=1}^N \in \mathcal{D}$, i.e. that the semiflow $\varphi_\ep$ generated by \eqref{eq:Pep} admits the desired Morse decomposition. 
We will repeat the proof of Proposition \ref{prop:1.2}. {To do that, we need to show (i) $\cup_{u_0 \in K^+} \omega(u_0) \subset \{E_i\}_{i=1}^{N+1}$; and that (ii) there is no cycle of fixed points.} 

To prove (i), it remains to consider a trajectory starting at some $u_0 \in K^+ \setminus (\textup{Int}\,K^+)$. By the strong maximum principle, it either enters the $\textup{Int}\,K^+$ for all $t>0$, or there is at least one component that is identically zero for all $t>0$. In the first case, the trajectory also converges to $E_1$. In the second case, 
it suffices to repeat the proofs for a suitable subset $(\tilde{d}_j)$ of $(d_i)$, to deduce again the convergence to the equilibrium $E_{i_0}$,  where $i_0$ is the smallest integer such that the $i_0$-th component of $u_0$ is non-zero. This proves (i). Next, we prove (ii), i.e. there is no cycle of fixed points. Indeed, if $E_i$ is chained to $E_j$, then necessarily $i > j$. 

It therefore follows from {\cite[Theorem 3.2 and Remark 4.6]{Hirsch2001}} that any compact internally chain transitive set is an equilibrium point. Since any omega (resp. alpha) limit set is internally chain transitive, it can only be one of the $E_i$.

In conclusion, for any choice of diffusion rates $(d_i)_{i=1}^N$ that is sufficiently close to $(\hat{d}_k)$ in the Hausdorff sense, 
the set of equilibria with the obvious ordering gives a Morse decomposition of the dynamics of \eqref{eq:Pep}. This means $(d_i)_{i=1}^N \in \mathcal{D}$.
\end{proof}

\section{Conclusion}\label{sec:7}

In his seminal paper \cite{Hastings1983}, Hastings showed that for two competing species  that are ecologically identical but having distinct diffusion rates, the slower disperser can invade the faster disperser when rare, but not vice versa. Later, Dockery et al. \cite{Dockery1998} proved that the slower disperser always competitively excludes the faster disperser, regardless of initial conditions, and conjectured that the same is true for any number of species.

In this paper, we show that for any number of competing species which are ecologically identical and having distinct diffusion rates $\{d_i\}_{i=1}^N$, 
there are choices of $\{d_i\}_{i=1}^N$ for which the slowest disperser is able to competitively exclude the remainder of the species. In fact, the choices of such diffusion rates is open in the space of finite sets of $\mathbb{R}_+$ endowed with the Hausdorff topology. Our result provides some evidence in the affirmative direction regarding the conjecture by Dockery et al. in \cite{Dockery1998}.

\appendix

\section{The normalized principal Floquet bundle}\label{sec:B}

\subsection{Existence and uniqueness results}

Let $\Omega \subset \mathbb{R}^n$ be a smooth bounded domain, 
and consider the linear parabolic operator of non-divergence form:
\begin{equation}\label{eq:Lt}
\partial_t\psi - \mathcal{L}_t \psi = \partial_t \psi-  a_{ij}(x,t) \partial^2_{x_ix_j} \psi - b_j (x,t)\partial_{x_j} \psi - c(x,t) \psi,
\end{equation}
{with oblique boundary condition} 
\begin{equation}\label{eq:Bt}
\mathcal{B}_t\psi = p_i(x,t) \partial_{x_i}\psi + p_0(x,t) \psi, 
\end{equation}
where repeated indices are summed from $1$ to $n$; the coefficients $a_{ij},b_j,c,p_i$ are continuous in $x,t$ and satisfy for some $\Lambda>1$ 
\begin{equation}\label{eq:boundM2} 
\begin{cases}
\frac{1}{\Lambda} |\xi|^2\leq a_{ij} (x,t)\xi_i \xi_j\leq \Lambda |\xi|^2  \quad  \text{ for }x \in \Omega,\, t \in \mathbb{R}, \,\xi \in \mathbb{R}^n,\\
\inf\limits_{\partial\Omega\times\mathbb{R}}p_0(x,t) \geq 0 \quad \text{ and }\quad \inf\limits_{\partial\Omega\times\mathbb{R}}\nu_i(x)p_i(x,t) >0 
\end{cases}
\end{equation}
where $(\nu(x))_{i=1}^n$ is the outward unit normal vector on $\partial\Omega$. 


We recall the existence and uniqueness of the normalized principal Floquet bundle.

\begin{theorem}\label{thm:A1}
Suppose 
$$
a_{ij},b_j,c \in C^{\beta,\beta/2}(\overline\Omega\times\mathbb{R})\quad \text{ and }\quad p_i \in C^{1+\beta, (1+\beta)/2}(\partial\Omega\times\mathbb{R}).$$ Then
there exists a unique pair  $(\psi_1(x,t),H_1(t)) \in C^{2+\beta, 1+\beta/2}(\overline\Omega \times \mathbb{R}) \times C^{\beta/2}(\mathbb{R})$ satisfying, in classical sense,
\begin{equation}\label{eq:bundlen}
\left\{
\begin{array}{ll}
\partial_t \psi -\mathcal{L}_t\psi = H(t)\psi  &\text{ for }x \in \Omega,\, t \in \mathbb{R},\\
{\mathcal{B}_t}\psi(x,t) = 0 &\text{ for }x \in \partial \Omega, \, t \in \mathbb{R},\\
{\int_\Omega \psi(x,t)\,dx = 1} & \text{ for  } t\in \mathbb{R},\\
\psi(x,t) >0   &\text{ for }x \in \Omega,\, t \in \mathbb{R}.
\end{array}
\right.
\end{equation}
Furthermore, there exists $C$ independent of $t$ such that
\begin{equation}\label{eq:expbound}
\frac{1}{C} \leq \psi_1(x,t) \leq {C}  \quad \text{ for all }x \in \Omega,\, t\in \mathbb{R}.
\end{equation}
\end{theorem}
\begin{proof}[Proof of Theorem \ref{thm:A1}]
By \cite[Theorem 2.1(iii) and Corollary 2.4]{Mierczynski} (which used on the abstract theory in \cite{Polacikbd}), the problem
\begin{equation}\label{eq:bundle}
\left\{
\begin{array}{ll}
\partial_t \tilde{\psi} - \mathcal{L}_t\tilde\psi=0 &\text{ for }x \in \Omega,\, t \in \mathbb{R},\\
{\mathcal{B}_t} \tilde{\psi}(x,t) = 0 &\text{ for }x \in \partial \Omega, \, t \in \mathbb{R},\\
\int_\Omega {\tilde{\psi}(x,0)}\,dx = 1,  & \\
\tilde{\psi}(x,t) >0  & \text{ for }x \in \Omega,\, t \in \mathbb{R},
\end{array}
\right.
\end{equation}
has a unique positive solution $\tilde{\psi}(x,t)$. 
By the standard parabolic regularity theory, one can observe that 
$\tilde \psi \in C^{2 + \beta,1+\beta/2}_{loc}( \overline{\Omega} \times \mathbb{R} )$. 
 Furthermore, the uniform Harnack inequality \cite[Theorem 2.5]{Huska2006} holds, i.e. there exists $C$ such that
\begin{equation}\label{eq:lowe1}
\sup\limits_{x \in \Omega} \tilde\psi(x,t) \leq C\inf\limits_{x \in \Omega} \tilde\psi(x,t) \quad \text{ for all }t \in \mathbb{R}.
\end{equation}

We proceed to define the normalize principal Floquet bundle $(\psi_1,H_1)$ by 
$$
{H_1(t) :=  - \frac{d}{dt}\left[ \log \Big\| \tilde\psi(\cdot\,,t)\Big\|_{L^1(\Omega)}\right] = -\frac{\int_\Omega  \partial_t \tilde\psi \,dx}{ \int_\Omega \tilde\psi\,dx}}
$$
and 
\begin{equation}\label{eq:lowe12}
\psi_1(x,t) := \exp\left(\int_0^t {H_1(s)}\,ds\right) \tilde{\psi}(x,t).
\end{equation}
Then it is immediate that $H_1 \in C^{\beta/2}_{loc}(\mathbb{R})$ and $\psi_1 \in C^{2 + \beta, 1+\beta/2}_{loc}(\overline\Omega \times \mathbb{R})$ and that $(\psi_1,H_1)$ satisfies \eqref{eq:bundlen}. To conclude the proof, it remains to show that
\begin{equation}\label{eq:hhold}
\|H_1\|_{C^{\beta/2}(\mathbb{R})} \leq C \quad \text{ and } \quad \|\psi_1\|_{C^{2 + \beta, 1+\beta/2}(\overline\Omega \times \mathbb{R})}\leq C.
\end{equation}

\begin{claim}\label{claim:5a}
There exists $C$ independent of $t_0 \in \mathbb{R}$ such that 
\begin{equation}\label{eq:lowe2a}
\sup_{\Omega \times [t_0,t_0+1]} \tilde\psi(x,t ) \leq e^{\|c\|_\infty} \sup_{x \in \Omega} \tilde\psi(x,t_0).
\end{equation}
\end{claim}
Fix $t_0 \in \mathbb{R}$, observe that $\psi^*(x,t):= e^{\|c\|_\infty t}\sup_\Omega \tilde\psi(\cdot,t_0)$ and $\tilde\psi(x,t)$ forms a pair of super and sub-solutions to \eqref{eq:bundle}. Thus \eqref{eq:lowe2a} follows by comparison.

\begin{claim}\label{claim:5b}
There exists $C$ independent of $t_0 \in \mathbb{R}$ such that 
\begin{equation}\label{eq:lowe2b}
\sup_{\Omega \times [t_0,t_0+1]} \tilde\psi(x,t ) \leq C \inf_{x \in \Omega} \tilde\psi(x,t_0+1).
\end{equation}
\end{claim}
Fix an $x_0 \in \Omega$, 
by the usual parabolic Harnack inequality \cite[Theorem 2.2]{Fabes1999}, there exists $C_1>0$, $m \in \mathbb{N}$ independent of $t_0 \in \mathbb{R}$ such that $\tilde\psi(x_0,t_0) \leq C_1  \tilde\psi(x_0,t_0+1/m)$. Taking $C_2 = (C_1)^m$, we have \begin{equation}\label{eq:lowe2c}
 \tilde\psi(x_0,t_0) \leq C_2  \tilde\psi(x_0,t_0+1)
\end{equation}
Combining \eqref{eq:lowe1} and \eqref{eq:lowe2c}, we have $$
\sup_\Omega \tilde\psi(\cdot,t_0) \leq C_3\inf_\Omega \tilde\psi(\cdot,t_0+1).
$$
Now, using also \eqref{eq:lowe2a}, we obtain \eqref{eq:lowe2b}.

{By parabolic estimates, there exists $C$ independent of $t \in \mathbb{R}$ such that
\begin{align}
\|\tilde\psi\|_{C^{\beta,\beta/2}(\overline{\Omega} \times [t-1/2,t])} +  \|\partial_t \tilde\psi\|_{C^{\beta,\beta/2}(\overline{\Omega} \times [t-1/2,t])} &\leq C \|\tilde\psi\|_{L^{\infty}(\Omega \times (t-1,t))}. \notag\\
\end{align}
Combining with \eqref{eq:lowe2b}, we have
$$
\|\tilde\psi\|_{C^{\beta,\beta/2}(\overline{\Omega} \times [t-1/2,t])} +  \|\partial_t \tilde\psi\|_{C^{\beta,\beta/2}(\overline{\Omega} \times [t-1/2,t])} \leq C \int_{\Omega} \tilde\psi(x,t)\,dx
$$
for some constant $C$ that is independent of $t \in \mathbb{R}$. In particular, if we define 
$$
F(t) := - \int_\Omega \partial_t \tilde\psi(x,t) \,dx \quad\text{  and }\quad G(t) := \int_\Omega {\tilde\psi}(x,t)\,dx,
$$
then there is $C$ independent of $t$ such that}
$$
{|F(t)|} + \frac{|F(t) - F(s)|}{|t-s|^{\beta/2}}  + \frac{|G(t) - G(s)|}{|t-s|^{\beta/2}} \leq C G(t) \quad \text{ for }s \in [t-1/2,t). 
$$
Since $H_1(t) = F(t)/G(t)$, we obtain $\|H_1\|_{C(\mathbb{R})} \leq C$ and 
$$
\frac{|H_1(t) - H_1(s)|}{|t-s|^{\beta/2}} \leq \frac{1}{G(t)}\frac{|F(t) - F(s)|}{|t-s|^{\beta/2}}  +  \frac{|F(s)|}{G(s)G(t)}\frac{|G(t) - G(s)|}{|t-s|^{\beta/2}}   \leq C 
$$
for $s \in [t-1/2,t)$. 
This proves the first half of \eqref{eq:hhold}.

Next, it follows from \eqref{eq:lowe1} and \eqref{eq:lowe12} that
$$
\psi_1(y,t) \leq C  \psi_1(x,t) \quad \text{ for all }x,y \in \Omega,\, t \in \mathbb{R}.
$$
where $C$ is independent of $x,y,t$. By fixing $x,t$ and integrate over $y \in \Omega$, 
$$
1 = \int_\Omega \psi_1(y,t)\,dy \leq C |\Omega| \psi_1(x,t) \quad \text{ for all }x\in\Omega,\, t\in \mathbb{R},
$$
where use used the normalization $\int_\Omega \psi_1(y,t)\,dy\equiv 1$. This
proves the lower bound of
\eqref{eq:expbound}. The upper bound can be proved in a similar manner.

Finally, the second half of \eqref{eq:hhold} follows by applying Schauder's estimate on \eqref{eq:bundlen}, using $H_1\in C^{\beta/2}(\mathbb{R})$ (the first half of \eqref{eq:hhold}) and $\|\psi_1\|_{L^\infty(\Omega\times\mathbb{R})} \leq C$ (from \eqref{eq:expbound}).
\end{proof}


\subsection{The decomposition and exponential separation}\label{sec:a.2}
For fixed $\beta \in (0,1)$ and given functions $a_{ij}, b_j, c\in C^{\beta,\beta/2}(\overline{\Omega}\times \mathbb{R})$,  and $p \in C^{1+\beta,(1+\beta)/2}(\overline{\Omega}\times \mathbb{R})$ let $(\psi_1(x,t), H_1(t))$ be given as in Theorem \ref{thm:A1}. Consider the non-autonomous problem
\begin{equation}\label{eq:w}
\left\{\begin{array}{ll}
\partial_t u  - \mathcal{L}_t u  = H_1(t)u  &\text{ for }x \in \Omega,\, t \geq s,\\
{\mathcal{B}_t} u = 0 &\text{ for }x \in \partial \Omega,\, t \geq s,\\
u(x,s) = u_0(x) &\text{ for }x \in \Omega.
\end{array}\right.
\end{equation}
In the notation of \cite[Ch. 6]{Lunardi},  we will henceforth regard \eqref{eq:w} as a non-autononmous problem and 
denote the corresponding evolution operator by $U(t,s)$, such that for $t \geq s$, $u(x,t)= U(t,s)[u_0](x)$ is the unique solution to \eqref{eq:w}.  It is immediate that for the normalized principal Floquet bundle $\psi_1$, we have
\begin{equation}\label{eq:aw}
U(t,s)[\psi_1(\cdot, s)] = \psi_1(\cdot, t) \quad \text{ whenever }t \geq s.
\end{equation}
Define,
for each $t\in\mathbb{R}$, $X^1(t):= \text{span}\,\{\psi_1(\cdot,t)\},$
and
$$
X^2(t):=\big\{ u_0 \in L^2(\Omega):\,  U(\tilde t,t)[u_0](x) \,\text{ has a zero in }\Omega\text{ for all }\tilde t \in (t,\infty)\big\}
$$
 Then 
 $X^1(t)$ and $X^2(t)$ are forward-invariant under $U(t,s)$, i.e.
 $$
 U(t,s)(X^1(s))  = X^1(t)\quad \text{ and }\quad  U(t,s)(X^2(s))  \subseteq X^2(t) \quad \text{ for }t \geq s.
 $$
Also, it follows by \cite{Huska2006,Huska2007} that 
\begin{equation}\label{eq:decompose}
L^2(\Omega) = X^1(t) \oplus X^2(t) \quad \text{ for each }t \in \mathbb{R}.
\end{equation}
By \eqref{eq:expbound} and the normalization $\|U(t,s)\psi_1(\cdot,s)\|_{L^1(\Omega)} = \|\psi_1(\cdot,s)\|_{L^1(\Omega)}=1$, there is $C$ such that
$$
\|U(t,s)\psi_1(\cdot,s)\|_{L^2(\Omega)} \geq C \|\psi_1(\cdot,s)\|_{L^2(\Omega)} \quad \text{ for any }t \geq s.
$$
It then follows from \cite[Theorem 2.1]{Huska2006} that
 there are constants $C,\gamma>0$ independent of time, such that 
\begin{equation}\label{eq:exposeparate}
\|U(t,s)v_0\|_{L^2(\Omega)} \leq C e^{-\gamma(t-s)}\|v_0\|_{L^2(\Omega)} \quad \text{ for }t \geq s\text{ and }v_0 \in X^2(s).
\end{equation}

Let $\beta>0$ be as given at the beginning of the appendix. In the following we will impose one of the following additional assumptions:
\begin{itemize}
    \item[\rm(H1)] The coefficients $a_{ij}, p_i$ are independent of $t$, and $a_{ij} \in C^{\beta}(\overline\Omega)$, $b_j,c \in C^{\beta,\beta/2}(\overline\Omega\times\mathbb{R})$, $p_i \in C^{1+\beta}(\partial\Omega)$; or
    \item[]
    \item[\rm(H2)] $a_{ij},b_j,c \in C^{\beta,(1+\beta)/2}(\overline\Omega\times\mathbb{R})$, 
    $p_i \in C^{1+\beta,(1+\beta)/2}(\partial\Omega\times\mathbb{R}) \cap C^{2,0}(\partial\Omega\times\mathbb{R})$.
\end{itemize}
We recall the following optimal regularity result concerning \eqref{eq:w}.
\begin{lemma}  \label{lem:optimalregular} 
Assume $a_{ij},b_j,c$ satisfies \eqref{eq:boundM2} and one of {\rm(H1)} or {\rm(H2)}.
{For each $T>0$ and $0<\alpha<1$, there exists a constant ${C}_T$ independent of $s \in \mathbb{R}$ and $u_0 \in C^\alpha(\overline\Omega)$ such that the solution  $u$ to the non-autonomous problem \eqref{eq:w} satisfies}
$$
\|u\|_{C^{\alpha,\alpha/2}(\overline\Omega \times [s, s+T])} \leq C_T \|u_0\|_{C^\alpha(\overline\Omega)}.
$$
\end{lemma}
\begin{proof}
Let the sectorial operator $\mathcal{L}_t$ and boundary operator $\mathcal{B}_t$ be given by \eqref{eq:Lt} and \eqref{eq:Bt} respectively. For each $t$, the domain of $\mathcal{L}_t$ is given by
\begin{align*}
D(\mathcal{L}_t)&=\left\{ \phi \in \bigcap_{p >1} W^{2,p}(\Omega): \mathcal{L}_t\phi \in C(\bar\Omega),\,\, \text{ and }\,\,\mathcal{B}_t \phi \big|_{\partial\Omega} = 0  \right\}\\
&=\left\{ \phi \in \bigcap_{p >1} W^{2,p}(\Omega): a_{ij}(\cdot,t)\partial^2_{x_ix_j}\phi \in C(\bar\Omega),\,\, \text{ and }\,\,\mathcal{B}_t \phi \big|_{\partial\Omega} = 0  \right\},
\end{align*}
which can vary with time. 
By \cite[Theorem 6.2]{Acquistapace1987a} (see also \cite[Thm. 3.1.30]{Lunardi}), 
\begin{equation}\label{eq:interpol}
(C(\overline\Omega);D(\mathcal{L}_t))_{\alpha/2,\infty} = C^{\alpha}(\overline\Omega) \quad \text{ for }0 < \alpha < 1,
\end{equation}
where for $\theta\in (0,1)$, we denote $(X,Y)_{\theta,\infty}$ to be the real interpolation space between $X \supset Y$ with its usual norm (cf. \cite{Lions1959} and \cite[Sect. 1.2.2]{Lunardi}).

Under assumption {\rm(H1)}, then $\mathcal{B}_t$ and thus the domain of 
$\mathcal{L}_t$, given by
$$
D(\mathcal{L}_t)=\left\{ \phi \in \bigcap_{p >1} W^{2,p}(\Omega): a_{ij}(\cdot)\partial^2_{x_ix_j}\phi \in C(\bar\Omega),\,\, \text{ and }\,\,  p_0\phi + p_i\partial_{x_i}\phi \big|_{\partial\Omega} = 0  \right\}
$$
is independent of $t$. 
In this case, it follows from \cite[(6.1.19) and Corollary 6.1.9(iii)]{Lunardi} 
that $t\mapsto U(t,s)[u_0]$ belongs to $L^\infty([s,s+T];C^{\alpha}(\bar\Omega)) \cap C^{\alpha/2}([s,s+T];C^0(\overline\Omega))$. Hence, $u(x,t) = U(t,s)[u_0](x) \in C^{\alpha,\alpha/2}(\overline\Omega)$.

Next, we consider the case {\rm(H2)} In this case, $D(\mathcal{L}_t)$ is time-varying. However, by \eqref{eq:interpol},  the interpolation space $(C(\overline\Omega), D(\mathcal{L}_t))_{\alpha/2,\infty}$ remains constant in time for each $0< \alpha<1$.
In this case, 
the method due to Acquistapace et al. \cite{Acquistapace1988} proves the existence of an evolution operator $U(t,s)$ in the Banach space $X=C(\overline\Omega)$.
Let $u_0 \in C^{\alpha}(\bar\Omega)$ be given for some $0<\alpha<1$. Then, \cite[Thm. 4.1(iii) and Thm. 4.2(iii)]{Acquistapace1988} say 
that $t\mapsto U(t,s)[u_0]$ belongs to $L^\infty([s,s+T];C^{\alpha}(\bar\Omega)) \cap C^{\alpha/2}([s,s+T];C^0(\overline\Omega))$. Hence, $u(x,t) = U(t,s)[u_0](x) \in C^{\alpha,\alpha/2}(\overline\Omega)$. We remark that the hypotheses I and II of \cite{Acquistapace1988} are verified in  \cite[Example 2]{Acquistapace1986} and \cite[Theorem 7.9]{Acquistapace1987}. This is where the additional H\"{o}lder  regularity of $a_{ij},b_j,c$ in time, with an exponent greater than $1/2$, is used.
\end{proof}
 Next, we discuss time and space regularity of the two projections $P^i(t)$.
\begin{lemma}\label{lem:A.13}
Assume $a_{ij},b_j,c$ satisfies \eqref{eq:boundM2} and one of {\rm(H1)} or {\rm(H2)}.
{For $i=1,2$ and $t \in \mathbb{R}$, let $P^i(t):L^2(\Omega) \to L^2(\Omega)$ be the projection operator corresponding to the decomposition given in \eqref{eq:decompose}.
\begin{itemize}
    \item[(a)] There exists a constant $\hat{C}_1$ independent of $t$ such that for $i=1,2$,
    $$
            \|P^1(t)[\hat f]\|_{C^\beta(\overline\Omega)} + \|P^2(t)[\hat f]\|_{L^\infty(\overline\Omega)}   \leq \hat{C}_1 \|\hat f\|_{L^\infty(\Omega)} 
            $$
    for every $\hat f \in L^\infty(\Omega)$.
    \item[(b)] There exists a constant $\hat{C}_2$ independent of $t_0 \in \mathbb{R}$ such that 
    $$
    \|P^1(t)[\hat f] - P^1(t_0)[\hat f]\|_{C^0(\overline\Omega)} \leq \hat{C}_2 \|\hat f\|_{C^\beta(\overline\Omega)}|t-t_0|^{\beta/2}
    $$
    for all $t \in [t_0, t_0 + 1]$ and all $\hat f \in C^\beta(\overline\Omega)$.
    \item[(c)] For $f \in C^{\beta,\beta/2}(\overline\Omega \times \mathbb{R})$ and  $i=1,2$, the mapping 
    $$
    (x,t) \mapsto P^i(t)[f(\cdot,t)](x) \quad \text{ is bounded in } C^{\beta,\beta/2}(\overline\Omega \times \mathbb{R}).
    $$
\end{itemize}}
\end{lemma}
\begin{proof}
For assertion (a), observe from the identity $P^2(t)[\hat f] = \hat f - P^1(t)[\hat f]$ that 
$$
\|P^2(t)[\hat f]\|_{L^\infty(\Omega)} \leq \|\hat f\|_{L^\infty(\Omega)} + \|P^1(t)[\hat f]\|_{L^\infty(\Omega)},
$$
so it remains to show 
\begin{equation}\label{eq:A131}
    \|P^1(t)[\hat f]\|_{C^\beta(\overline\Omega)} \leq C \|\hat f\|_{L^\infty(\Omega)}.
\end{equation}
Given $\hat{f}(x) \in L^\infty(\Omega)$, we write
\begin{equation}\label{eq:A114}
    P^1(t)[\hat f](x) = \sigma(t) \psi_1(x,t) \quad \text{ for }x \in \Omega,\, t\in \mathbb{R}.
\end{equation}
Since $\psi_1 \in C^{2+\beta,1+\beta/2}(\overline\Omega \times \mathbb{R})$, it suffices to show that
\begin{equation}\label{eq:A.113}
    \sup\limits_{t \in \mathbb{R}}|\sigma(t)| \leq C\|\hat f\|_{L^\infty(\Omega)}.
\end{equation}
with some $C$ independent of $t$. 
Since $P^2(t)[\hat f]$ has to change sign on $\Omega$, 
$$
\hat f(x) - \sigma(t) \psi_1(x,t) = P^2(t)[\hat f] \quad \text{ changes sign on }\Omega.
$$
In view of the bound \eqref{eq:expbound},  we deduce \eqref{eq:A.113}. 
This proves \eqref{eq:A131}.

For assertion (b), we let $\hat{f} \in C^\beta(\overline\Omega)$ be given. 
Recall that $U(t,s)$ denotes the evolution operator to \eqref{eq:w}. 
Recalling \eqref{eq:A114}, then
$$
\hat{f}(x) =P^1(t_0)[\hat f](x) + P^2(t_0)[\hat f](x) = \sigma(t_0)\psi_1(t_0,x)  + P^2(t_0)[\hat f](x),
$$
then
the forward-invariance of $X^1(t)$ and $X^2(t)$ under $U(t,s)$ implies
$$
U(t,t_0)[\hat{f}](x) = \sigma(t_0)\psi_1(t,x) + P^2(t)[U(t,t_0)[\hat{f}]](x),
$$
where we used $U(t,t_0)P^2(t_0)=P^2(t)U(t,t_0)$.
Taking $P^1(t)$ on both sides of the above, we get
\begin{equation}\label{eq:A.16b}
P^1(t)[U(t,t_0)[\hat{f}]](x) = \sigma(t_0) \psi_1(t,x).\end{equation} 
Hence,
\begin{align*}
&\quad \|P^1(t)[\hat f] - P^1(t_0)[\hat f]\|_{C^0(\overline\Omega)} \\
&\leq   \|P^1(t)[\hat f] - P^1(t)[U(t,t_0)[\hat{f}]]\|_{C^0(\overline\Omega)} +  \|\sigma(t_0) (\psi_1(t,\cdot) - \psi_1(t_0,\cdot))\|_{C^0(\overline\Omega)} \\
&=\|P^1(t)[\hat f(\cdot) - U(t,t_0)[\hat{f}]]\|_{C^0(\overline\Omega)} +  |\sigma(t_0)| \cdot  \|\psi_1(t,\cdot) - \psi_1(t_0,\cdot)\|_{C^0(\overline\Omega)}\\
&\leq C \|\hat f(\cdot) - U(t,t_0)[\hat{f}]\|_{L^\infty(\Omega)} + C\|\hat f\|_{L^\infty(\Omega)} |t-t_0|^{\beta/2} \\
&\leq C\left\{[U(t,t_0)[\hat{f}]]_{C^{\beta,\beta/2}(\overline\Omega \times [t_0,t_0+1])}  + \|\hat f\|_{L^\infty(\Omega)}\right\} |t-t_0|^{\beta/2}\\
&\leq C \|\hat f\|_{C^\beta(\overline\Omega)} |t-t_0|^{\beta/2},
\end{align*}
where the first inequality follows from \eqref{eq:A114} and \eqref{eq:A.16b},  the second inequality from 
 \eqref{eq:A.113} and  $\psi_1 \in C^{\beta,\beta/2}(\overline\Omega \times \mathbb{R}) \subset C^{\beta/2}(\mathbb{R}; C^{0}(\overline\Omega))$, while
the last inequality follows from {Lemma \ref{lem:optimalregular}.}

Next, we show assertion (c). Let $f \in C^{\beta,\beta/2}(\overline\Omega \times \mathbb{R})$ be given, 
and define 
$$
g(x,t):= P^1(t)[f(\cdot,t)](x)
$$
By part (a), we see that $g \in L^\infty( \mathbb{R}; C^\beta(\overline\Omega))$. It remains to establish the H\"{o}lder regularity in time, i.e. $g \in C^{\beta/2}( \mathbb{R}; C^0(\overline\Omega))$. On the one hand, for $t - t_0 \geq 1$, we have
$$
\|g(\cdot,t) - g(\cdot,t_0)\|_{C^0(\overline\Omega)} \leq 2\|g\|_{L^\infty( \mathbb{R}; C^\beta(\overline\Omega))} |t-t_0|^{\beta/2}.
$$
On the other hand, for $0 \leq t - t_0 < 1$, we apply (a) and (b) to get
\begin{align*}
  &\quad   \|g(\cdot,t) - g(\cdot,t_0)\|_{C^0(\overline\Omega)} \\&\leq    \|P^1(t)[f(\cdot,t)-f(\cdot,t_0)]\|_{C^0(\overline\Omega)} + \|P^1(t)[f(\cdot,t_0)] - P^1(t_0)[f(\cdot,t_0)]\|_{C^0(\overline\Omega)}\\
    &\leq C\| f(\cdot,t)-f(\cdot,t_0)\|_{L^\infty(\Omega)} + C \|f(\cdot,t_0)\|_{C^\beta(\overline\Omega)} |t-t_0|^{\beta/2}\\
    &\leq C |f|_{C^{\beta,\beta/2}(\overline\Omega \times \mathbb{R})} |t-t_0|^{\beta/2}. 
\end{align*}
Hence, $g \in C^{\beta/2}(\mathbb{R}; C^{0}(\overline\Omega)) \cap L^\infty(\mathbb{R}; C^\beta(\overline\Omega))$ and thus $g \in C^{\beta,\beta/2}(\overline\Omega \times \mathbb{R})$.

Finally, since $P^2(t)[f(\cdot,t)](x) = f(x,t) - g(x,t)$, we conclude that $(x,t)\mapsto P^2(t)[f(\cdot,t)](x)$ belongs to $C^{\beta,\beta/2}(\overline\Omega \times \mathbb{R})$ as well.
\end{proof}

\subsection{Smooth dependence on coefficients}\label{sec:a.3}

Define the spaces
$$
X^{(1)}_{\rm coeff}= \left\{(a_{ij},b_j,c,p_j) \in C^{\beta}(\overline\Omega)\times  [C^{\beta,\beta/2}(\overline\Omega\times\mathbb{R})]^2\times C^{1+\beta}(\partial\Omega): \, \eqref{eq:boundM2}\text{ holds}\right\}
$$
and $X^{(2)}_{\rm coeff}$ be the corresponding subset of $[C^{\beta,(1+\beta)/2}(\overline\Omega\times\mathbb{R})]^3\times C^{2,(1+\beta)/2}(\partial\Omega\times\mathbb{R})$. They correspond to hypotheses {\rm(H1)} and {\rm(H2)}. 
Proposition \ref{prop:A2'} is a particular case of the following result.
\begin{proposition}\label{prop:A2}
For $k=1$ or $2$, the following mapping is smooth:
\begin{align*}
    X^{(k)}_{\rm coeff}\qquad &\rightarrow \quad C^{2+\beta,1+\beta/2}(\overline\Omega \times \mathbb{R})\times C^{\beta/2}( \mathbb{R})\\
    \mathcal{A}=(a_{ij},b_j, c,p_j) \qquad &\mapsto \quad  \qquad  (\psi_1,H_1).
\end{align*}
\end{proposition}

\begin{proof}
We denote $\psi_1=\psi_1(x,t;\mathcal{A})$ and $H_1 = H_1(t;\mathcal{A})$
to stress the dependence of the normalized principal Floquet bundle on the coefficients $\mathcal{A}=(a_{ij},b_j,c,p_j)$ of $\mathcal{L}_t,\mathcal{B}_t$.
In the following, we will prove the smooth dependence. (We remark that the continuous dependence of $(\psi_1,H_1)$ on $\mathcal{A}$ follows from the  uniqueness
of the pair and standard parabolic regularity; see \cite{Huska2006} for details.) 
Consider the mapping 
\begin{align*}
&\mathcal{F}\, : \, C_{\mathcal{B}}^{2+\beta,1+\beta/2}(\overline{\Omega} \times \mathbb{R})\,\times\, C^{\beta/2}(\mathbb{R})  \,\times\, {X_{\rm coeff}}^{(k)}\\
&\qquad \qquad \qquad \qquad \qquad \qquad \qquad \qquad \,\, \longrightarrow \,\, C^{\beta,\beta/2}(\overline{\Omega} \times \mathbb{R}) \times C^{1+\beta/2}(\mathbb{R})
\end{align*}
that is defined by
$$
\mathcal{F}(\psi(x,t), H(t), \mathcal{A}):= 
\left( \begin{array}{c}
\partial_t \psi(x,t) - \mathcal{L}\psi(x,t) -H(t)\psi(x,t) \\
{\int_\Omega \psi(x,t)\,dx -1}
\end{array}
\right),
$$
where
$$
C_{\mathcal{B}}^{2+\beta,1+\beta/2}(\overline{\Omega} \times \mathbb{R}) = \{ u \in C^{2+\beta,1+\beta/2}(\overline{\Omega} \times \mathbb{R}):\, \mathcal{B}_t u = 0 \,\,\text{ in }\partial \Omega \times \mathbb{R}\}.
$$
Then, for each fixed $\mathcal{A}=(a_{ij},b_j,c,p_j) \in X_{\rm coeff}^{(k)}$,  $$\mathcal{F}(\psi_1(\cdot,\cdot;\mathcal{A}), H_1(\cdot;\mathcal{A}), \mathcal{A}) = 0.$$ To prove the smooth dependence on $\mathcal{A}$, it suffices to show that 
\begin{equation}\label{eq:1329}
D_{(\psi,H)}\mathcal{F}=D_{(\psi,H)}\mathcal{F}(\psi_1(\cdot,\cdot;\mathcal{A}), H_1(\cdot;\mathcal{A}), \mathcal{A}),
\end{equation} 
as a mapping from $C_{\mathcal{B}}^{2+\beta,1+\beta/2}(\overline{\Omega}\times \mathbb{R})\,\times\, C^{\beta/2}(\mathbb{R})$ to 
$C^{\beta,\beta/2}(\overline{\Omega} \times \mathbb{R})\,\times\, C^{1+\beta/2}(\mathbb{R})
$, is invertible. To this end, given $(f(x,t),G(t)) \in C^{\beta,\beta/2}(\overline{\Omega}\times\mathbb{R})\,\times\, C^{1+\beta/2}(\mathbb{R})$, we need to prove the existence and uniqueness of $(w(x,t),Y(t))$ in the class $C_{\mathcal{B}}^{2+\beta,1+\beta/2}(\overline{\Omega}\times \mathbb{R})\,\times\, C^{\beta/2}(\mathbb{R})$ such that
\begin{equation}\label{eq:IFT_cond}
\left\{
\begin{array}{ll}
\partial_t w - \mathcal{L} w - H_1w - Y(t)\psi_1 = f(x,t) &\text{ for }t \in \mathbb{R},\, x \in \Omega,\\
{\int_\Omega w(x,t) }\,dx = G(t) &\text{ for }t \in \mathbb{R},
\end{array}
\right.
\end{equation}
where $H_1= H_1(t;\mathcal{A})$ and $\psi_1 = \psi_1(x,t;\mathcal{A})$. 
First, we show the existence. We start by choosing $w^\perp$ as
$$
w^\perp(x,t) = \int_{-\infty}^t U(t,\tau)[P^2(s)f(\cdot,\tau)]\,d\tau,
$$
where, for $i=1,2$, $P^i:L^2(\Omega) \to L^2(\Omega)$ is the projection operator corresponding to the decomposition given in \eqref{eq:decompose}, {as discussed in Lemma \ref{lem:A.13}}.
We claim that $w^\perp$ is well defined. Indeed, by
\eqref{eq:exposeparate}, 
we have 
\begin{align}
\sup_{t \in \mathbb{R}}\|w^\perp(\cdot,t)\|_{L^2(\Omega)} 
&\leq \sup_{t \in \mathbb{R}}\left\| \int_{-\infty}^t U(t,\tau)[
P^2(\tau)
[f(\cdot, \tau)]]\, d\tau\right\|_{L^2(\Omega)}  \notag \\
&\leq C\sup_{t \in \mathbb{R}}\int_{-\infty}^t e^{-\gamma(t-s)}\left\| f(\cdot, \tau)\right\|_{{L^\infty}(\Omega)}\, d\tau  \notag \\
&\leq C\sup_{t \in \mathbb{R}}\|f(\cdot, t)\|_{{L^\infty}(\Omega)} < \infty
\label{eq:B.7b} 
\end{align}
where we used {\eqref{eq:exposeparate} and Lemma \ref{lem:A.13}(a) to get the second inequality.} 
Moreover,  since {$w^\perp$} defines an entire solution of $\partial_t w^\perp - \mathcal{L} w^\perp - H_1(t) w^\perp = P^2(t)[f(\cdot,t)]$ with homogeneous oblique boundary condition, {and that the right hand side $P^2(t)[f(\cdot,t)]$ is bounded in $C^{\beta,\beta/2}(\overline\Omega \times \mathbb{R})$ (by Lemma \ref{lem:A.13}(c)),} it follows by parabolic regularity estimates \cite[Chapter IV, Theorem 4.30]{Lieberman} that $w^\perp\in C_{\mathcal{B}}^{2+\beta,1+\beta/2}(\overline\Omega\times \mathbb{R})$. Next, define $w \in C_{\mathcal{B}}^{2+\beta,1+\beta/2}(\overline\Omega\times \mathbb{R})$ to be
$$
w(x,t) = w^\perp(x,t) + \left[ - { \int_\Omega w^\perp(y,t)\,dy} + G(t) \right] \psi_1(x,t).
$$
Then $w$ satisfies the second part of \eqref{eq:IFT_cond}. Moreover, we have
\begin{align*}
&    \partial_t w - \mathcal{L}w - H_1(t)w \\
&    = P^2(t)[f(\cdot,t)]+ \left\{ - \frac{d}{dt} \left[{ \int_\Omega w^\perp(y,t)\,dy} \right] + G'(t)\right\}\psi_1(x,t).
\end{align*}
It therefore suffices to choose $Y(t)$ such that 
$$
Y(t)\psi_1(x,t) + P^1(t)[f(\cdot,t)] = \left\{ - \frac{d}{dt}\left[{ \int_\Omega w^\perp(y,t)\,dy} \right] + G'(t)\right\}\psi_1(x,t).
$$
{Note that $Y \in C^{\beta/2}( \mathbb{R})$ by using the $C^{2+\beta,1+\beta/2}$ regularity of $w^\perp$ and Lemma \ref{lem:A.13}(c).} 
Then $(w,Y) \in C_{\mathcal{B}}^{2+\beta,1+\beta/2}(\overline\Omega\times \mathbb{R})\times C^{\beta,\beta/2}(\overline\Omega\times \mathbb{R})$ satisfies \eqref{eq:IFT_cond}. This proves existence.

For the uniqueness, set $f=0$ and $G=0$, then using the the variation of constants formula for $\partial_t w - \mathcal{L} w - H_1(t) w =  Y(t) \psi_1(x,t)$, we get
\begin{align*}
w(\cdot,t) &=  U(t,s)w(\cdot,s) + \int_{s}^t U(t,\tau)[Y(\tau)\psi_1(\cdot,\tau)]\,d\tau\\
&= U(t,s)w(\cdot,s) + \int_s^t Y(\tau) \{U(t,\tau) [\psi_1(\cdot,\tau)]\}\,d\tau\\
&= U(t,s)w(\cdot,s) + \int_s^t Y(\tau) \psi_1(\cdot,t) \,d\tau\\
&= U(t,s)w(\cdot,s) + \left[\int_s^t Y(\tau)  \,d\tau\right] \psi_1(\cdot,t) \quad \text{ for }t>s,
\end{align*}
where we used \eqref{eq:aw} for the third equality. Hence, we deduce  
\begin{equation}\label{eq:varconst}
w(\cdot,t)= U(t,s)w(\cdot,s) + \left[\int_s^t Y(\tau)  \,d\tau\right] \psi_1(\cdot,t)  \quad \text{ for any } t > s.
\end{equation}
Next,  apply the projection $P^2(t)$ on both sides of \eqref{eq:varconst},
$$
P^2(t)[w(\cdot,t)] = P^2(t)[U(t,s)w(\cdot,s)] = U(t,s)P^2(s)[w(\cdot,s)].
$$
provided $t,s \in \mathbb{R}$ and $t > s$. This implies
\begin{align}
\Big\|P^2(t)[w(\cdot,t)]\Big\|_{L^2(\Omega)}  
&\leq C e^{-\gamma(t-s)} \Big\|P^2(s)[w(\cdot,s)]\Big\|_{L^2(\Omega)}  \notag\\
&\leq C e^{-\gamma(t-s)}\Big\|w(\cdot,s)\Big\|_{{L^\infty}(\Omega)}  \leq  C e^{-\gamma(t-s)},\label{eq:iiii}
\end{align}
where we used \eqref{eq:exposeparate} for the first inequality, {Lemma \ref{lem:A.13}(a) for the second one}, 
and the fact that $w \in C^{2+\beta,1+\beta/2}(\overline\Omega \times \mathbb{R})$ for the third one. Letting $s \to -\infty$ in \eqref{eq:iiii}, we deduce that $P^2(t)[w(\cdot,t)] = 0$ for each $t\in\mathbb{R}$. Hence, $w(\cdot,t) \in X^1(t)$ and thus $w(\cdot,t) = \sigma(t)\psi_1(\cdot,t)$ for some function $\sigma(t)$. Now, using $G(t) \equiv 0$, the  second equation in \eqref{eq:IFT_cond} gives 
$$
{0 = \int_\Omega w(x,t) \,dx =  \sigma(t) \int_\Omega \psi_1(x,t)\,dx =\sigma(t)}
$$
for each $t\in \mathbb{R}$. This implies $w(x,t) \equiv 0$. Substituting into \eqref{eq:varconst}, we have
$$
\int_s^t Y(\tau)\,d\tau = 0 \quad \text{ for any }t,s \in \mathbb{R}, \, t > s,
$$
which means $Y(t) \equiv 0$ as well. 
This proves uniqueness. 

Having shown that $D_{(\psi,H)} \mathcal{F}$ given in \eqref{eq:1329} is an isomorphism, we may apply the implicit function theorem to conclude the smooth dependence of the normalized principal Floquet bundle $(\psi_1(x,t),H_1(t))$ on the coefficients $\mathcal{A}$. This concludes the proof.
\end{proof}

\end{document}